\documentclass[12pt]{amsart}
\usepackage{amssymb}
\usepackage[all]{xy}

\newtheorem{theorem}{Theorem}[section]
\newtheorem{definition}[theorem]{Definition}
\newtheorem{proposition}[theorem]{Proposition}

\begin{document}

\title[Noncommutative spheres]{Quantum isometries, noncommutative spheres, and related integrals}

\author{Teodor Banica}
\address{T.B.: Department of Mathematics, Cergy-Pontoise University, 95000 Cergy-Pontoise, France. {\tt teodor.banica@u-cergy.fr}}


\begin{abstract}
The sphere $S^{N-1}_\mathbb R$ has a half-liberated analogue $S^{N-1}_{\mathbb R,*}$, and a free analogue $S^{N-1}_{\mathbb R,+}$. This is a presentation of the construction and main properties of these noncommutative spheres, $S^{N-1}_\mathbb R\subset S^{N-1}_{\mathbb R,*}\subset S^{N-1}_{\mathbb R,+}$, and of their quantum isometry groups.
\end{abstract}

\maketitle

\tableofcontents

\section*{Introduction}

A recent discovery, from \cite{bsp}, \cite{bve}, states that under the ``strongest possible axioms'', there are only three orthogonal quantum groups, $O_N\subset O_N^*\subset O_N^+$. These quantum groups correspond to three ``main'' noncommutative spheres, $S^{N-1}_\mathbb R\subset S^{N-1}_{\mathbb R,*}\subset S^{N-1}_{\mathbb R,+}$, introduced and studied in \cite{bgo}, and in a number of subsequent papers. 

We discuss here these constructions and results, by using a ``sphere-first'' approach, which is perhaps more natural. All the needed preliminaries are included. 

This is based on lecture notes from a minicourse given at the Summer school ``Topological quantum groups'', Bedlewo 2015. It is a pleasure to thank Uwe Franz, Adam Skalski and Piotr So\l tan for the invitation, and for the nice  organization of the meeting.

\section{Noncommutative spheres}

We adhere here to the general principle that ``the noncommutative spaces are the abstract duals of the operator algebras''. Our starting point will be:

\begin{definition}
A $C^*$-algebra is a complex algebra with unit, with an involution $*$ and a norm $||.||$, such that the Cauchy sequences converge, and such that $||aa^*||=||a||^2$.
\end{definition}

The basic example is the matrix algebra $M_N(\mathbb C)$, with involution $(M^*)_{ij}=\overline{M}_{ji}$, and with the norm $||M||=\sup_{||x||=1}||Mx||$. More generally, we have as example $B(H)$, the algebra of bounded operators $T:H\to H$ on a Hilbert space $H$, with involution given by $<T^*x,y>=<x,Ty>$, and with norm $||T||=\sup_{||x||=1}||Tx||$. The GNS theorem states that any $C^*$-algebra appears as closed $*$-subalgebra of some $B(H)$.

Another key example is $C(X)$, the algebra of continuous functions on
a compact space $X$, with involution $f^*(x)=\overline{f(x)}$, and with norm $||f||=\sup_{x\in X}|f(x)|$. The Gelfand theorem states that any
commutative $C^*$-algebra is of this form. To be more precise, given a commutative $C^*$-algebra $A$, the underlying compact space $X=Spec(A)$ is the set of characters $\chi:A\to\mathbb C$, with topology making the evaluation maps continuous. 

In view of Gelfand's theorem, we can formulate:

\begin{definition}
The category of noncommutative compact spaces is the category of the $C^*$-algebras, with the arrows reversed. Given a noncommutative compact space $X$, coming from a $C^*$-algebra $A$, we write $A=C(X)$, and $X=Spec(A)$.
\end{definition}

Observe that the category of usual compact spaces embeds into the category of noncommutative compact spaces. More precisely, a compact space $X$ corresponds to the noncommutative space associated to the algebra $A=C(X)$. In addition, in this situation, $X$ can be recovered as a Gelfand spectrum, $X=Spec(A)$.

Consider now the standard sphere, $S^{N-1}_\mathbb R=\{x\in\mathbb R^N|\sum_ix_i^2=1\}$. In order to discuss its noncommutative analogues, we must first understand the associated algebra $C(S^{N-1}_\mathbb R)$. The result here, coming from the Gelfand theorem, is as follows:

\begin{proposition}
We have the presentation result
$$C(S^{N-1}_\mathbb R)=C^*_{comm}\left(x_1,\ldots,x_N\Big|x_i=x_i^*,\sum_ix_i^2=1\right)$$
where by $C^*_{comm}$ we mean universal commutative $C^*$-algebra.
\end{proposition}

\begin{proof}
We have a morphism from right to left, which by the Stone-Weierstrass theorem is surjective. In the other sense now, the universal algebra on the right being commutative, by the Gelfand theorem it must be of the form $C(X)$, for a certain compact space $X$. The coordinate functions $x_i$ provide us with an embedding $X\subset\mathbb R^N$, and then the quadratic condition $\sum_ix_i^2=1$ shows that we have $X\subset S^{N-1}_\mathbb R$. Thus, we have as well a morphism from left to right. Since the two morphisms that we constructed map standard coordinates to standard coordinates, they are inverse to each other, and we are done.
\end{proof}

The idea now is to replace the commutation relations $ab=ba$ between the standard coordinates on $S^{N-1}_\mathbb R$ by some weaker relations. A first choice is that of using no relations at all. A second choice, coming from the easy quantum group philosophy \cite{bsp}, is that of using the ``half-commutation'' relations $abc=cba$. So, let us formulate:

\begin{definition}
Associated to any $N\in\mathbb N$ is the following universal $C^*$-algebra:
$$C(S^{N-1}_{\mathbb R,+})=C^*\left(x_1,\ldots,x_N\Big|x_i=x_i^*,\sum_ix_i^2=1\right)$$
The quotient of this algebra by the relations $x_ix_jx_k=x_kx_jx_i$ is denoted $C(S^{N-1}_{\mathbb R,*})$.
\end{definition}

Observe that the above two algebras are indeed well-defined, because the quadratic relations $\sum_ix_i^2=1$ show that we have $||x_i||\leq1$, for any $C^*$-norm. Thus the biggest $C^*$-norm is bounded, and the enveloping $C^*$-algebras are well-defined.

Given a noncommutative compact space $X=Spec(A)$, its classical version $X_{class}$, which is a usual compact space, is by definition the Gelfand spectrum $X_{class}=Spec(A/I)$, where $I\subset A$ is the commutator ideal. With this convention, we have:

\begin{proposition}
We have inclusions of noncommutative compact spaces 
$$S^{N-1}_\mathbb R\subset S^{N-1}_{\mathbb R,*}\subset S^{N-1}_{\mathbb R,+}$$
and $S^{N-1}_\mathbb R$ is the classical version of both the spaces on the right.
\end{proposition}

\begin{proof}
Since the commutation relations $ab=ba$ imply the half-commutation relations $abc=cba$, we have quotient maps $C(S^{N-1}_{\mathbb R,+})\to C(S^{N-1}_{\mathbb R,*})\to C(S^{N-1}_\mathbb R)$, which correspond to inclusions as above. As for the last assertion, this follows from Proposition 1.3.
\end{proof}

As already mentioned, the definition of $S^{N-1}_{\mathbb R,*}$ is quite tricky. Our claim is that, under strong axioms, this sphere is the unique intermediate one $S^{N-1}_\mathbb R\subset S\subset S^{N-1}_{\mathbb R,+}$. This will be discussed later on, in section 2 below. For the moment, let us just record an elementary result, which can serve as a temporary motivation for the study of our 3 spheres:

\begin{proposition}
The closed subspace $S^{(k)}\subset S^{N-1}_{\mathbb R,+}$ obtained by imposing the relations $a_1\ldots a_k=a_k\ldots a_1$ to the standard coordinates of $S^{N-1}_{\mathbb R,+}$ is as follows:
\begin{enumerate}
\item At $k=1$ we have $S^{(k)}=S^{N-1}_{\mathbb R,+}$.

\item At $k=2,4,6,\ldots$ we have $S^{(k)}=S^{N-1}_\mathbb R$.

\item At $k=3,5,7,\ldots$ we have $S^{(k)}=S^{N-1}_{\mathbb R,*}$.
\end{enumerate}
\end{proposition}

\begin{proof}
Since the relations $ab=ba$ imply the relations $a_1\ldots a_k=a_k\ldots a_1$ for $k\geq2$, we have $S^{(2)}\subset S^{(k)}$ for $k\geq2$. It is also elementary to check that the relations $abc=cba$ imply the relations $a_1\ldots a_k=a_k\ldots a_1$ for $k\geq3$ odd, so $S^{(3)}\subset S^{(k)}$ for $k\geq3$ odd.

Our claim now is that we have $S^{(k+2)}\subset S^{(k)}$, for any $k\geq2$. In order to prove this, we must show that the relations $a_1\ldots a_{k+2}=a_{k+2}\ldots a_1$ between $x_1,\ldots,x_N$ imply the relations $a_1\ldots a_k=a_k\ldots a_1$ between $x_1,\ldots,x_N$. But this holds indeed, because:
\begin{eqnarray*}
x_{i_1}\ldots x_{i_{k+2}}=x_{i_{k+2}}\ldots x_{i_1}
&\implies&x_{i_1}\ldots x_{i_k}x_j^2=x_j^2x_{i_k}\ldots x_{i_1}\\
&\implies&\sum_jx_{i_1}\ldots x_{i_k}x_j^2=\sum_jx_j^2x_{i_k}\ldots x_{i_1}\\
&\implies&x_{i_1}\ldots x_{i_k}=x_{i_k}\ldots x_{i_1}
\end{eqnarray*}

Summing up, we have proved that we have inclusions $S^{(2)}\subset\ldots\subset S^{(6)}\subset S^{(4)}\subset S^{(2)}$ and $S^{(3)}\subset\ldots\subset S^{(7)}\subset S^{(5)}\subset S^{(3)}$, and this gives the result.
\end{proof}

Given a closed subspace $S\subset S^{N-1}_{\mathbb R,+}$, the associated ``noncommutative cube'' $K\subset S$ is obtained by setting $C(K)=C(S)/<x_i^2=\frac{1}{N}>$. As a basic example, for the usual sphere $S=S^{N-1}_\mathbb R$ we obtain the usual cube, $K=\{x\in\mathbb R^N|x_i=\pm\frac{1}{\sqrt{N}}\}$. Also, given a discrete group $\Gamma$, we denote by $\widehat{\Gamma}$ the noncommutative space dual to $A=C^*(\Gamma)$. We have:

\begin{proposition}
The noncommutative cubes associated to the $3$ spheres are
$$\begin{matrix}
S^{N-1}_\mathbb R&\subset&S^{N-1}_{\mathbb R,*}&\subset&S^{N-1}_{\mathbb R,+}\\
\\
\cup&&\cup&&\cup\\
\\
\widehat{\mathbb Z_2^N}&\subset&\widehat{\mathbb Z_2^{\circ N}}&\subset&\widehat{\mathbb Z_2^{*N}}
\end{matrix}$$
where $\mathbb Z_2^{\circ N}=<g_1,\ldots,g_N|g_i^2=1,g_ig_jg_k=g_kg_jg_i>$. All these inclusions are proper.
\end{proposition}

\begin{proof}
Let us first compute the noncommutative cube $K^{N-1}_{\mathbb R,+}$ associated to $S^{N-1}_{\mathbb R,+}$. Since the relations $x_i^2=\frac{1}{N}$ imply the quadratic condition $\sum_ix_i^2=1$, we have:
$$C(K^{N-1}_{\mathbb R,+})=C^*\left(x_1,\ldots,x_N\Big|x_i=x_i^*,x_i^2=\frac{1}{N}\right)$$

On the other hand, consider the group $\mathbb Z_2^{*N}=<g_1,\ldots,g_N|g_i^2=1>$. Since $g_i=g_i^{-1}$, we have $g_i=g_i^*$ in the corresponding group algebra, which is therefore given by:
$$C^*(\mathbb Z_2^{*N})=C^*\left(g_1,\ldots,g_N\Big|g_i=g_i^*,g_i^2=1\right)$$

Thus we have an isomorphism $C(K^{N-1}_{\mathbb R,+})\simeq C^*(\mathbb Z_2^{*N})$, given by $x_i=g_i/\sqrt{N}$, and at the level of the correspoding noncommutative spaces we obtain $K^{N-1}_{\mathbb R,+}\simeq\widehat{\mathbb Z_2^{*N}}$. This establishes the inclusion on the right, and the other two vertical inclusions are now clear, too.

Finally, since $\widehat{\mathbb Z_2^N}\subset S^{N-1}_\mathbb R$ is proper, so are the other two vertical inclusions. Also, since the quotient maps $\mathbb Z_2^{*N}\to\mathbb Z_2^{\circ N}\to\mathbb Z_2^N$ are both proper, so are both the horizontal inclusions on the bottom, and hence the horizontal inclusions on top as well.
\end{proof}

The above occurrence of $\mathbb Z_2^{\circ N}$ is quite of interest. As a first observation, if $h_1,\ldots,h_N$ are the standard generators of $\mathbb Z^N$, then we have a matrix model, as follows:
$$C^*(\mathbb Z_2^{\circ N})\to M_2(C^*(\mathbb Z^N))\quad:\quad g_i\to\begin{pmatrix}0&h_i\\ h_i^{-1}&0\end{pmatrix}$$ 

Indeed, the matrices $G_i$ on the right satisfy the relations $G_i=G_i^*$, $G_i^2=1$ and $G_iG_jG_k=G_kG_jG_i$, which define the group algebra on the left.

The point now is that $S^{N-1}_{\mathbb R,*}$ itself has a similar model, as follows:

\begin{proposition}
We have a morphism of $C^*$-algebras
$$C(S^{N-1}_{\mathbb R,*})\to M_2(C(S^{N-1}_\mathbb C))\quad:\quad x_i\to\begin{pmatrix}0&z_i\\ \bar{z}_i&0\end{pmatrix}$$ 
where $S^{N-1}_\mathbb C=\{z\in\mathbb C^N|\sum_i|z_i|^2=1\}$ is the unit complex sphere.
\end{proposition}

\begin{proof}
We have to prove that the matrices $X_i$ on the right satisfy the defining relations for $S^{N-1}_{\mathbb R,*}$. These matrices are indeed self-adjoint, and their squares sum up to 1:
$$\sum_iX_i^2=\sum_i\begin{pmatrix}0&z_i\\ \bar{z}_i&0\end{pmatrix}^2=\sum_i\begin{pmatrix}|z_i|^2&0\\0&|z_i|^2\end{pmatrix}=\begin{pmatrix}1&0\\0&1\end{pmatrix}$$

Regarding now the half-commutation relations, observe that we have:
$$X_iX_jX_k=\begin{pmatrix}0&z_i\\ \bar{z}_i&0\end{pmatrix}\begin{pmatrix}0&z_j\\ \bar{z}_j&0\end{pmatrix}\begin{pmatrix}0&z_k\\ \bar{z}_k&0\end{pmatrix}=\begin{pmatrix}0&z_i\bar{z}_jz_k\\ \bar{z}_iz_j\bar{z}_k&0\end{pmatrix}$$

Since this quantity is symmetric in $i,k$, this gives the result.
\end{proof}

It is standard to prove that the above model of $C^*(\mathbb Z_2^{\circ N})$ is faithful. The same happens for $C(S^{N-1}_{\mathbb R,*})$, and we will prove this, after developing a number of useful tools.

We recall that $P^{N-1}_\mathbb R$ is the space of lines in $\mathbb R^N$ passing through the origin. We have a quotient map $S^{N-1}_\mathbb R\to P^{N-1}_\mathbb R$, which produces an embedding $C(P^{N-1}_\mathbb R)\subset C(S^{N-1}_\mathbb R)$, and the image of this embedding is the algebra generated by the variables $p_{ij}=x_ix_j$. 

In general now, based on this observation, we can formulate:

\begin{definition}
The projective version of $S\subset S^{N-1}_{\mathbb R,+}$ is the quotient space $S\to PS$ determined by the fact that $C(PS)\subset C(S)$ is the subalgebra generated by $p_{ij}=x_ix_j$.
\end{definition}

It follows from the above discussion that we have $PS^{N-1}_\mathbb R=P^{N-1}_\mathbb R$, and our goal now is to compute the projective version of the remaining 2 spheres. As a first observation, the projective version of $S^{N-1}_{\mathbb R,*}$ is a classical space, because its coordinates commute:
$$abcd=cbad=cdab$$

We will prove that this space is the complex projective one, $PS^{N-1}_{\mathbb R,*}=P^{N-1}_\mathbb C$.

For this purpose, we will need a functional analytic description of $P^{N-1}_\mathbb R,P^{N-1}_\mathbb C$. The result here, which is similar to the one in Proposition 1.3 above, is as follows:

\begin{proposition}
We have presentation results as follows,
\begin{eqnarray*}
C(P^{N-1}_\mathbb C)&=&C^*_{comm}\left((p_{ij})_{i,j=1,\ldots,N}\Big|p=p^*=p^2,Tr(p)=1\right)\\
C(P^{N-1}_\mathbb R)&=&C^*_{comm}\left((p_{ij})_{i,j=1,\ldots,N}\Big|p=\bar{p}=p^*=p^2,Tr(p)=1\right)
\end{eqnarray*}
where by $C^*_{comm}$ we mean as usual universal commutative $C^*$-algebra.
\end{proposition}

\begin{proof}
We use the fact that $P^{N-1}_\mathbb C,P^{N-1}_\mathbb R$ are respectively the spaces of rank one projections in $M_N(\mathbb C),M_N(\mathbb R)$. With this picture in mind, the first formula is clear from the Gelfand theorem. Also, since $P^{N-1}_\mathbb R\subset P^{N-1}_\mathbb C$ appears by restricting attention to the matrices which are real, the relation to be added is $p=\bar{p}$, and this gives the second formula. 
\end{proof}

The above result suggests the following definition:

\begin{definition}
Associated to any $N\in\mathbb N$ is the following universal algebra,
$$C(P^{N-1}_+)=C^*\left((p_{ij})_{i,j=1,\ldots,N}\Big|p=p^*=p^2,Tr(p)=1\right)$$
whose abstract spectrum is called ``free projective space''.
\end{definition}

Observe that we have embeddings of noncommutative spaces $P^{N-1}_\mathbb R\subset P^{N-1}_\mathbb C\subset P^{N-1}_+$, and that the complex projective space $P^{N-1}_\mathbb C$ is the classical version of $P^{N-1}_+$.

We have the following result, first established in \cite{bgo}:

\begin{theorem}
The projective versions of the $3$ spheres are given by
$$\begin{matrix}
S^{N-1}_\mathbb R&\subset&S^{N-1}_{\mathbb R,*}&\subset&S^{N-1}_{\mathbb R,+}\\
\\
\downarrow&&\downarrow&&\downarrow\\
\\
P^{N-1}_\mathbb R&\subset&P^{N-1}_\mathbb C&\subset&\mathcal P^{N-1}_+
\end{matrix}$$
where $\mathcal P^{N-1}_+$ is a certain noncommutative compact space, contained in $P^{N-1}_+$.
\end{theorem}

\begin{proof}
The assertion at left is true by definition. For the assertion at right, we have to prove that the variables $p_{ij}=x_ix_j$ over the free sphere $S^{N-1}_{\mathbb R,+}$ satisfy the defining relations for $C(P^{N-1}_+)$ from Definition 1.11, and the verification here goes as follows:
\begin{eqnarray*}
(p^*)_{ij}&=&p_{ji}^*=(x_jx_i)^*=x_ix_j=p_{ij}\\
(p^2)_{ij}&=&\sum_kp_{ik}p_{kj}=\sum_kx_ix_k^2x_j=x_ix_j=p_{ij}\\
Tr(p)&=&\sum_kp_{kk}=\sum_kx_k^2=1
\end{eqnarray*}

Regarding now the middle assertion, stating that we have $PS^{N-1}_{\mathbb R,*}=P^{N-1}_\mathbb C$: 

``$\subset$'' follows from the relations $abc=cba$, which imply $abcd=cbad=cbda$. Indeed, this shows that $PS^{N-1}_{\mathbb R,*}$ is classical, and so $PS^{N-1}_{\mathbb R,*}\subset(P^{N-1}_+)_{class}=P^{N-1}_\mathbb C$.

``$\supset$'' follows by using the model in Proposition 1.8. Indeed, the representation there maps $p_{ij}\to P_{ij}=diag(z_i\bar{z}_j,\bar{z}_iz_j)$, and so maps $<p_{ij}>\to <P_{ij}>=C(P^{N-1}_\mathbb C)$.
\end{proof}

Let us prove now that the matrix model in Proposition 1.8 is faithful. As a warm-up here, we first prove the result in the group case. The statement, from \cite{bve}, is:

\begin{proposition}
We have an embedding of $C^*$-algebras
$$C^*(\mathbb Z_2^{\circ N})\subset M_2(C^*(\mathbb Z^N))\quad:\quad g_i\to\begin{pmatrix}0&h_i\\ h_i^{-1}&0\end{pmatrix}$$ 
where $h_1,\ldots,h_N$ are the standard generators of $\mathbb Z^N$.
\end{proposition}

\begin{proof}
Consider the crossed product $\mathbb Z^N\rtimes\mathbb Z_2$, with the group $\mathbb Z_2=<\tau>$ acting on $\mathbb Z^N$ via $\tau\cdot x=x^{-1}$. Our claim is that we have a group embedding, as follows:
$$\mathbb Z_2^{\circ N}\subset\mathbb Z^N\rtimes\mathbb Z_2\quad:\quad g_i\to (h_i,\tau)$$

Indeed, the elements $G_i=(h_i,\tau)$ are reflections, and satisfy $abc=cba$. Regarding now the injectivity, the point here is that each word $w$ in the kernel must be such that each $g_i$ appears an equal number of times at odd and even positions in $w$. Thus, $w=1$.

We therefore have an embedding $C^*(\mathbb Z_2^{\circ N})\subset C^*(\mathbb Z^N\rtimes\mathbb Z_2)$, and by composing with the standard embedding $C^*(\mathbb Z^N\rtimes\mathbb Z_2)\subset M_2(C^*(\mathbb Z^N))$, the result follows.
\end{proof}

In the sphere case now, we have the following result, from \cite{bic}:

\begin{theorem}
We have an embedding of $C^*$-algebras
$$C(S^{N-1}_{\mathbb R,*})\subset M_2(C(S^{N-1}_\mathbb C))\quad:\quad x_i\to\begin{pmatrix}0&z_i\\ \bar{z}_i&0\end{pmatrix}$$ 
where $S^{N-1}_\mathbb C=\{z\in\mathbb C^N|\sum_i|z_i|^2=1\}$ is the unit complex sphere.
\end{theorem}

\begin{proof}
As in the group case, this follows by using crossed products. To be more precise, our first claim is that we have an embedding of $C^*$-algebras, as follows:
$$C(S^{N-1}_{\mathbb R,*})\subset C(S^{N-1}_\mathbb C)\rtimes\mathbb Z_2\quad:\quad x_i\to z_i\otimes\tau$$

Indeed, the elements $X_i=z_i\otimes\tau$ are self-adjoint, satisfy $abc=cba$, and their squares sum up to 1. Regarding now the injectivity, this follows from Theorem 1.12, which shows that the morphism is injective on the subalgebra $C(P^{N-1}_\mathbb C)$. See \cite{bic}.

Now observe that we have as well an embedding as follows, where $f\to\widetilde{f}$ is the automorphism of $C(S^{N-1}_\mathbb C)$ induced by the conjugation of the coordinates, $z_i\to\bar{z}_i$:
$$C(S^{N-1}_\mathbb C)\rtimes\mathbb Z_2\subset M_2(C(S^{N-1}_\mathbb C))\quad:\quad f\otimes1\to\begin{pmatrix}f&0\\0&\widetilde{f}\end{pmatrix}\ ,\ f\otimes\tau\to\begin{pmatrix}0&f\\\widetilde{f}&0\end{pmatrix}$$

By composing with the embedding found above, the result follows.
\end{proof}

\section{Axiomatization, classification}

In this section we axiomatize our three spheres, $S^{N-1}_\mathbb R\subset S^{N-1}_{\mathbb R,*}\subset S^{N-1}_{\mathbb R,+}$. We already know from Proposition 1.6 that these are exactly the closed subspaces $S\subset S^{N-1}_{\mathbb R,+}$ which can be obtained by imposing relations of type $a_1\ldots a_k=a_k\ldots a_1$ to the standard coordinates of $S^{N-1}_{\mathbb R,+}$. We will improve here this result, by using ``arbitrary permutations''.

To be more precise, let us start with the following notion:

\begin{definition}
A monomial sphere is a subset $S\subset S^{N-1}_{\mathbb R,+}$ obtained via relations of type
$$x_{i_1}\ldots x_{i_k}=x_{i_{\sigma(1)}}\ldots x_{i_{\sigma(k)}},\ \forall (i_1,\ldots,i_k)\in\{1,\ldots,N\}^k$$
with $\sigma\in S_k$ being certain permutations, of variable size $k\in\mathbb N$.
\end{definition}

Observe that the basic 3 spheres are all monomial, with the permutations producing $S^{N-1}_\mathbb R,S^{N-1}_{\mathbb R,*}$ being the standard crossing and the half-liberated crossing:
$$
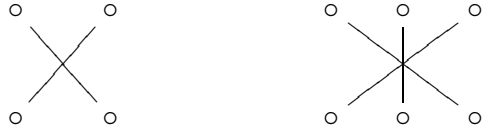
\xymatrix@R=10mm@C=8mm{\circ\ar@{-}[dr]&\circ\ar@{-}[dl]\\\circ&\circ}\qquad\qquad\qquad
\xymatrix@R=10mm@C=5mm{\circ\ar@{-}[drr]&\circ\ar@{-}[d]&\circ\ar@{-}[dll]\\\circ&\circ&\circ}$$ 

Here, and in what follows, we agree to represent the permutations $\sigma\in S_k$ by diagrams between two rows of $k$ points, acting by definition downwards.

Observe also that Proposition 1.6 reformulates as follows:

\begin{proposition}
The monomial spheres coming from mirroring permutations,
$$
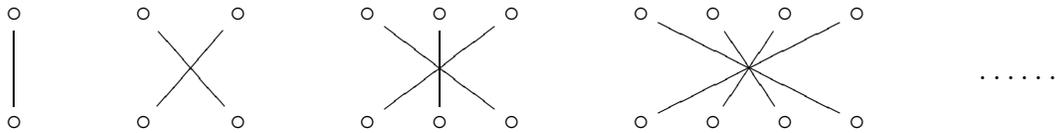
\xymatrix@R=10mm@C=8mm{\circ\ar@{-}[d]\\\circ}
\qquad\quad
\xymatrix@R=10mm@C=8mm{\circ\ar@{-}[dr]&\circ\ar@{-}[dl]\\\circ&\circ}\qquad\quad
\xymatrix@R=10mm@C=5mm{\circ\ar@{-}[drr]&\circ\ar@{-}[d]&\circ\ar@{-}[dll]\\\circ&\circ&\circ}
\qquad\quad
\xymatrix@R=10mm@C=5mm{\circ\ar@{-}[drrr]&\circ\ar@{-}[dr]&\circ\ar@{-}[dl]&\circ\ar@{-}[dlll]\\\circ&\circ&\circ&\circ}
\qquad\quad
\xymatrix@R=5mm@C=5mm{\\\ldots\ldots}
$$ 
are precisely the $3$ main spheres, $S^{N-1}_\mathbb R\subset S^{N-1}_{\mathbb R,*}\subset S^{N-1}_{\mathbb R,+}$.
\end{proposition}

\begin{proof}
This follows indeed from Proposition 1.6, because the relations $a_1\ldots a_k=a_k\ldots a_1$ used there are precisely those coming from mirroring permutations.
\end{proof}

We will prove in what follows that the basic 3 spheres are the only monomial ones. For this purpose, it is convenient to introduce the inductive limit $S_\infty=\bigcup_{k\geq0}S_k$, with the inclusions $S_k\subset S_{k+1}$ being given by $\sigma\in S_k\implies\sigma(k+1)=k+1$. In terms of elements of $S_\infty$, the definition of the monomial spheres reformulates as follows:

\begin{proposition}
The monomial spheres are the subsets $S\subset S^{N-1}_{\mathbb R,+}$ obtained via relations
$$x_{i_1}\ldots x_{i_k}=x_{i_{\sigma(1)}}\ldots x_{i_{\sigma(k)}},\ \forall (i_1,\ldots,i_k)\in\{1,\ldots,N\}^k$$
associated to certain elements $\sigma\in S_\infty$, where $k\in\mathbb N$ is such that $\sigma\in S_k$. 
\end{proposition}

\begin{proof}
We must prove that the relations $x_{i_1}\ldots x_{i_k}=x_{i_{\sigma(1)}}\ldots x_{i_{\sigma(k)}}$ are left unchanged when replacing $k\to k+1$. But this follows from $\sum_ix_i^2=1$, because:
\begin{eqnarray*}
x_{i_1}\ldots x_{i_k}x_{i_{k+1}}=x_{i_{\sigma(1)}}\ldots x_{i_{\sigma(k)}}x_{i_{k+1}}
&\implies&x_{i_1}\ldots x_{i_k}x_{i_{k+1}}^2=x_{i_{\sigma(1)}}\ldots x_{i_{\sigma(k)}}x_{i_{k+1}}^2\\
&\implies&\sum_{i_{k+1}}x_{i_1}\ldots x_{i_k}x_{i_{k+1}}^2=\sum_{i_{k+1}}x_{i_{\sigma(1)}}\ldots x_{i_{\sigma(k)}}x_{i_{k+1}}^2\\
&\implies&x_{i_1}\ldots x_{i_k}=x_{i_{\sigma(1)}}\ldots x_{i_{\sigma(k)}}
\end{eqnarray*}

Thus we can indeed ``simplify at right'', and this gives the result.
\end{proof}

In order to prove the uniqueness result, we use group theory methods. We call a subgroup $G\subset S_\infty$ filtered when it is stable under concatenation, in the sense that when writing $G=(G_k)$ with $G_k\subset S_k$, we have $\sigma\in G_k,\pi\in G_l\implies \sigma\pi\in G_{k+l}$. With this convention, each monomial sphere comes from a filtered group of permutations:

\begin{proposition}
The monomial spheres are the subsets $S_G\subset S^{N-1}_{\mathbb R,+}$ given by
$$C(S_G)=C(S^{N-1}_{\mathbb R,+})\Big/\Big<x_{i_1}\ldots x_{i_k}=x_{i_{\sigma(1)}}\ldots x_{i_{\sigma(k)}},\forall (i_1,\ldots,i_k)\in\{1,\ldots,N\}^k,\forall\sigma\in G_k\Big>$$
where $G=(G_k)$ is a filtered subgroup of $S_\infty=(S_k)$.
\end{proposition}

\begin{proof}
We know from Proposition 2.3 that the construction in the statement produces a monomial sphere. Conversely, given a monomial sphere $S\subset S^{N-1}_{\mathbb R,+}$, let us set:
$$G_k=\left\{\sigma\in S_k\Big|x_{i_1}\ldots x_{i_k}=x_{i_{\sigma(1)}}\ldots x_{i_{\sigma(k)}},\forall (i_1,\ldots,i_k)\in\{1,\ldots,N\}^k\right\}$$

With $G=(G_k)$ we have $S=S_G$, so it remains to prove that $G$ is a filtered group.

Since the relations $x_{i_1}\ldots x_{i_k}=x_{i_{\sigma(1)}}\ldots x_{i_{\sigma(k)}}$ can be composed and reversed, each $G_k$ follows to be stable under composition and inversion, and is therefore a group. 

Also, since the relations $x_{i_1}\ldots x_{i_k}=x_{i_{\sigma(1)}}\ldots x_{i_{\sigma(k)}}$ can be concatenated as well, our group $G=(G_k)$ is stable under concatenation, and we are done. 
\end{proof}

As an illustration, the groups $\{1\}\subset S_\infty$ produce the spheres $S^{N-1}_{\mathbb R,+}\supset S^{N-1}_\mathbb R$. In order to discuss now the half-liberated case, we will need:

\begin{proposition}
Let $S_\infty^*\subset S_\infty$ be the set of permutations having the property that when labelling cyclically the legs $\bullet\circ\bullet\circ\ldots$, each string joins a black leg to a white leg.
\begin{enumerate}
\item $S_\infty^*$ is a filtered subgroup of $S_\infty$, generated by the half-liberated crossing.

\item We have $S_{2k}^*\simeq S_k\times S_k$, and $S^*_{2k+1}\simeq S_k\times S_{k+1}$, for any $k\in\mathbb N$.
\end{enumerate}
\end{proposition}

\begin{proof}
The fact that $S_\infty^*$ is indeed a subgroup of $S_\infty$, which is filtered, is clear. Observe now that the half-liberated crossing has the ``black-to-white'' joining property:
$$
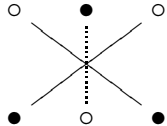
\xymatrix@R=10mm@C=5mm{\circ\ar@{-}[drr]&\bullet\ar@{.}[d]&\circ\ar@{-}[dll]\\\bullet&\circ&\bullet}$$ 

Thus this crossing belongs to $S_3^*$, and it is routine to check, by double inclusion, that the filtered subgroup of $S_\infty$ generated by it is the whole $S_\infty^*$.  Regarding now the last assertion, observe first that $S_3^*,S_4^*$ consist of the following permutations:
$$
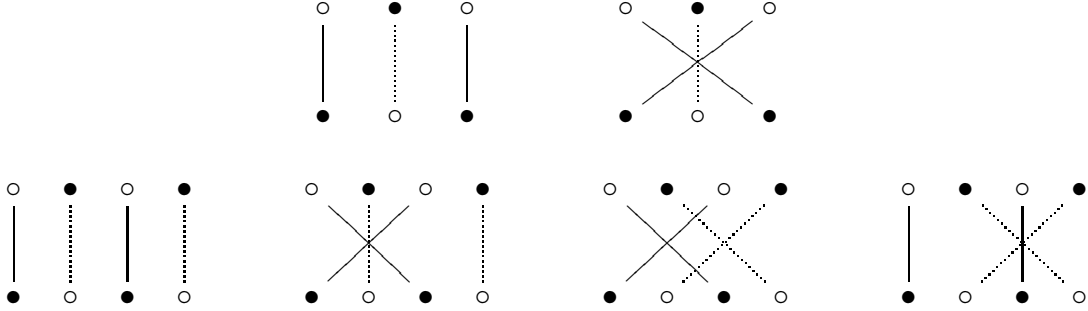
\xymatrix@R=10mm@C=5mm{\circ\ar@{-}[d]&\bullet\ar@{.}[d]&\circ\ar@{-}[d]\\\bullet&\circ&\bullet}\qquad\qquad 
\xymatrix@R=10mm@C=5mm{\circ\ar@{-}[drr]&\bullet\ar@{.}[d]&\circ\ar@{-}[dll]\\\bullet&\circ&\bullet}$$ 
\smallskip
$$\xymatrix@R=10mm@C=3mm{\circ\ar@{-}[d]&\bullet\ar@{.}[d]&\circ\ar@{-}[d]&\bullet\ar@{.}[d]\\\bullet&\circ&\bullet&\circ}\quad\qquad
\xymatrix@R=10mm@C=3mm{\circ\ar@{-}[drr]&\bullet\ar@{.}[d]&\circ\ar@{-}[dll]&\bullet\ar@{.}[d]\\\bullet&\circ&\bullet&\circ}\quad\qquad
\xymatrix@R=10mm@C=3mm{\circ\ar@{-}[drr]&\bullet\ar@{.}[drr]&\circ\ar@{-}[dll]&\bullet\ar@{.}[dll]\\\bullet&\circ&\bullet&\circ}\quad\qquad
\xymatrix@R=10mm@C=3mm{\circ\ar@{-}[d]&\bullet\ar@{.}[drr]&\circ\ar@{-}[d]&\bullet\ar@{.}[dll]\\\bullet&\circ&\bullet&\circ}$$

Thus we have $S_3^*=S_1\times S_2$ and $S_4^*=S_2\times S_2$, with the first component coming from dotted permutations, and with the second component coming from the solid line permutations. The same argument works in general, and gives the last assertion.
\end{proof}

Now back to the main 3 spheres, the result is as follows:

\begin{proposition}
The monomial spheres $S^{N-1}_\mathbb R\subset S^{N-1}_{\mathbb R,*}\subset S^{N-1}_{\mathbb R,+}$ come respectively from the filtered groups $S_\infty\supset S_\infty^*\supset\{1\}$.
\end{proposition}

\begin{proof}
This is clear by definition in the classical and in the free cases. In the half-liberated case, the result follows from Proposition 2.5 (1) above.
\end{proof}

Now back to the general case, consider a monomial sphere $S_G\subset S^{N-1}_{\mathbb R,+}$, with the filtered group $G\subset S_\infty$ taken to be maximal, as in the proof of Proposition 2.4. 

We have the following key observation:

\begin{proposition}
The filtered group $G\subset S_\infty$ associated to a monomial sphere $S\subset S^{N-1}_{\mathbb R,+}$ is stable under the following operations, on the corresponding diagrams:
\begin{enumerate}
\item Removing outer strings.

\item Removing neighboring strings.
\end{enumerate}
\end{proposition}

\begin{proof}
Both these results follow by using the quadratic condition:

(1) Regarding the outer strings, by summing over $a$, we have indeed:
$$Xa=Ya\implies Xa^2=Ya^2\implies X=Y$$
$$aX=aY\implies a^2X=a^2Y\implies X=Y$$

(2) Regarding the neighboring strings, once again by summing over $a$, we have:
$$XabY=ZabT\implies Xa^2Y=Za^2T\implies XY=ZT$$
$$XabY=ZbaT\implies Xa^2Y=Za^2T\implies XY=ZT$$

Thus $G=(G_k)$ has both the properties in the statement.
\end{proof}

We are now in position of stating and proving the axiomatization result regarding our 3 noncommutative spheres, which was recently obtained in \cite{bme}:

\begin{theorem}
The spheres $S^{N-1}_\mathbb R\subset S^{N-1}_{\mathbb R,*}\subset S^{N-1}_{\mathbb R,+}$ are the only monomial ones.
\end{theorem}

\begin{proof}
We will prove that the only filtered groups $G\subset S_\infty$ satisfying the conditions in Proposition 2.7 are $\{1\}\subset S_\infty^*\subset S_\infty$, correspoding to our 3 spheres. In order to do so, consider such a filtered group $G\subset S_\infty$, assumed to be non-trivial, $G\neq\{1\}$.

\underline{Step 1.} Our first claim is that $G$ contains a 3-cycle. For this purpose, we use a standard trick, stating that if $\pi,\sigma\in S_\infty$ have support overlapping on exactly one point, say $supp(\pi)\cap supp(\sigma)=\{i\}$, then the commutator $\sigma^{-1}\pi^{-1}\sigma\pi$ is a 3-cycle, namely $(i,\sigma^{-1}(i),\pi^{-1}(i))$. Indeed the computation of the commutator goes as follows:
$$
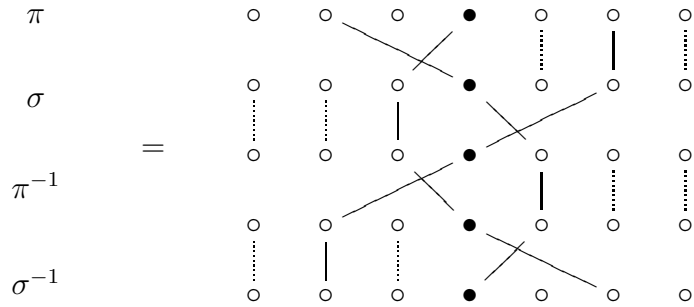
\xymatrix@R=7mm@C=5mm{\pi\\ \sigma\\ \pi^{-1}\\ \sigma^{-1}}\qquad
\xymatrix@R=6mm@C=5mm{\\ \\ =}\qquad
\xymatrix@R=5mm@C=5mm{
\circ&\circ\ar@{-}[drr]&\circ&\bullet\ar@{-}[dl]&\circ\ar@{.}[d]&\circ\ar@{-}[d]&\circ\ar@{.}[d]\\
\circ\ar@{.}[d]&\circ\ar@{.}[d]&\circ\ar@{-}[d]&\bullet\ar@{-}[dr]&\circ&\circ\ar@{-}[dll]&\circ\\
\circ&\circ&\circ\ar@{-}[dr]&\bullet\ar@{-}[dll]&\circ\ar@{-}[d]&\circ\ar@{.}[d]&\circ\ar@{.}[d]\\
\circ\ar@{.}[d]&\circ\ar@{-}[d]&\circ\ar@{.}[d]&\bullet\ar@{-}[drr]&\circ\ar@{-}[dl]&\circ&\circ\\
\circ&\circ&\circ&\bullet&\circ&\circ&\circ
}$$

Now let us pick a non-trivial element $\tau\in G$. By removing outer strings at right and at left we obtain permutations $\tau'\in G_k,\tau''\in G_s$ having a non-trivial action on their right/left leg, and by taking $\pi=\tau'\otimes id_{s-1},\sigma=id_{k-1}\otimes\tau''$, the trick applies.

\underline{Step 2.} Our second claim is $G$ must contain one of the following permutations:
$$
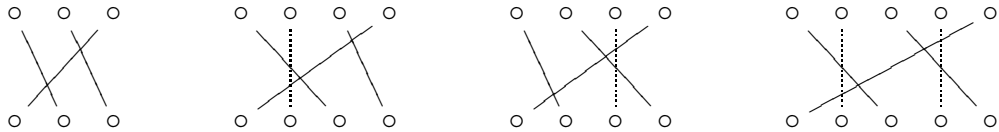
\xymatrix@R=10mm@C=2mm{
\circ\ar@{-}[dr]&\circ\ar@{-}[dr]&\circ\ar@{-}[dll]\\
\circ&\circ&\circ}\qquad\quad
\xymatrix@R=10mm@C=2mm{
\circ\ar@{-}[drr]&\circ\ar@{.}[d]&\circ\ar@{-}[dr]&\circ\ar@{-}[dlll]\\
\circ&\circ&\circ&\circ}\qquad\quad
\xymatrix@R=10mm@C=2mm{
\circ\ar@{-}[dr]&\circ\ar@{-}[drr]&\circ\ar@{.}[d]&\circ\ar@{-}[dlll]\\
\circ&\circ&\circ&\circ}\qquad\quad
\xymatrix@R=10mm@C=2mm{
\circ\ar@{-}[drr]&\circ\ar@{.}[d]&\circ\ar@{-}[drr]&\circ\ar@{.}[d]&\circ\ar@{-}[dllll]\\
\circ&\circ&\circ&\circ&\circ}$$

Indeed, consider the 3-cycle that we just constructed. By removing all outer strings, and then all pairs of adjacent vertical strings, we are left with these permutations.

\underline{Step 3.} Our claim now is that we must have $S_\infty^*\subset G$. Indeed, let us pick one of the permutations that we just constructed, and apply to it our various diagrammatic rules. From the first permutation we can obtain the basic crossing, as follows:
$$
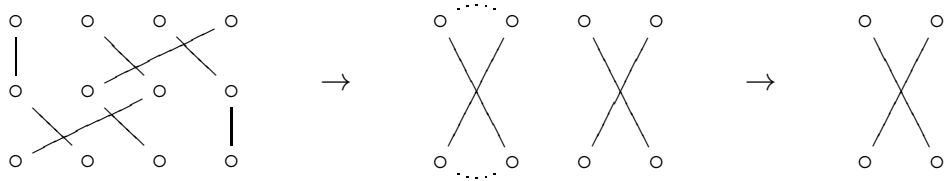
\xymatrix@R=5mm@C=5mm{
\circ\ar@{-}[d]&\circ\ar@{-}[dr]&\circ\ar@{-}[dr]&\circ\ar@{-}[dll]\\
\circ\ar@{-}[dr]&\circ\ar@{-}[dr]&\circ\ar@{-}[dll]&\circ\ar@{-}[d]\\
\circ&\circ&\circ&\circ}\qquad
\xymatrix@R=5mm@C=5mm{
\\ \to\\}\qquad
\xymatrix@R=6mm@C=5mm{
\circ\ar@{-}[ddr]\ar@/^/@{.}[r]&\circ\ar@{-}[ddl]&\circ\ar@{-}[ddr]&\circ\ar@{-}[ddl]\\
\\
\circ\ar@/_/@{.}[r]&\circ&\circ&\circ}\qquad
\xymatrix@R=5mm@C=5mm{
\\ \to\\}\qquad
\xymatrix@R=6mm@C=5mm{
\circ\ar@{-}[ddr]&\circ\ar@{-}[ddl]\\
\\
\circ&\circ}$$

Also, by removing a suitable $\slash\hskip-2.1mm\backslash$ shaped configuration, which is represented by dotted lines in the diagrams below, we can obtain the basic crossing from the second and third permutation, and the half-liberated crossing from the fourth permutation:
$$
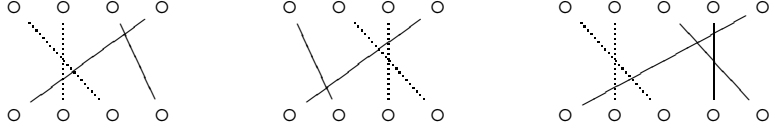
\xymatrix@R=10mm@C=2mm{
\circ\ar@{.}[drr]&\circ\ar@{.}[d]&\circ\ar@{-}[dr]&\circ\ar@{-}[dlll]\\
\circ&\circ&\circ&\circ}\qquad\quad
\xymatrix@R=10mm@C=2mm{
\circ\ar@{-}[dr]&\circ\ar@{.}[drr]&\circ\ar@{.}[d]&\circ\ar@{-}[dlll]\\
\circ&\circ&\circ&\circ}\qquad\quad
\xymatrix@R=10mm@C=2mm{
\circ\ar@{.}[drr]&\circ\ar@{.}[d]&\circ\ar@{-}[drr]&\circ\ar@{-}[d]&\circ\ar@{-}[dllll]\\
\circ&\circ&\circ&\circ&\circ}$$

Thus, in all cases we have a basic or half-liberated crossing, and so $S_\infty^*\subset G$.

\underline{Step 4.} Our last claim, which will finish the proof, is that there is no proper intermediate subgroup $S_\infty^*\subset G\subset S_\infty$. In order to prove this, observe that $S_\infty^*\subset S_\infty$ is the subgroup of  parity-preserving permutations, in the sense that ``$i$ even $\implies$ $\sigma(i)$ even''. 

Now let us pick an element $\sigma\in S_k-S_k^*$, with $k\in\mathbb N$. We must prove that the group $G=<S_\infty^*,\sigma>$ equals the whole $S_\infty$. In order to do so, we use the fact that $\sigma$ is not parity preserving. Thus, we can find $i$ even such that $\sigma(i)$ is odd. 

In addition, up to passing to $\sigma|$, we can assume that $\sigma(k)=k$, and then, up to passing one more time to $\sigma|$, we can further assume that $k$ is even.

Since both $i,k$ are even we have $(i,k)\in S_k^*$, and so $\sigma(i,k)\sigma^{-1}=(\sigma(i),k)$ belongs to $G$. But, since $\sigma(i)$ is odd, by deleting an appropriate number of vertical strings, $(\sigma(i),k)$ reduces to the basic crossing $(1,2)$. Thus $G=S_\infty$, and we are done.
\end{proof}

Summarizing, we have now a quite satisfactory axiomatization of our three spheres. We can axiomatize as well our noncommutative projective spaces, as follows:

\begin{definition}
A monomial projective space is a closed subset $P\subset P^{N-1}_+$ of the free projective space obtained via relations of type
$$p_{i_1i_2}\ldots p_{i_{k-1}i_k}=p_{i_{\sigma(1)}i_{\sigma(2)}}\ldots p_{i_{\sigma(k-1)}i_{\sigma(k)}},\ \forall (i_1,\ldots,i_k)\in\{1,\ldots,N\}^k$$
with $\sigma$ ranging over a certain subset of $\bigcup_{k\in2\mathbb N}S_k$, stable under $\sigma\to|\sigma|$.
\end{definition}

Observe the similarity with Definition 2.1. The only subtlety in the projective case is the stability under $\sigma\to|\sigma|$, which in practice means that if the above relation associated to $\sigma$ holds, then the following relation, associated to $|\sigma|$, must hold as well:
$$p_{i_0i_1}\ldots p_{i_ki_{k+1}}=p_{i_0i_{\sigma(1)}}p_{i_{\sigma(2)}i_{\sigma(3)}}\ldots p_{i_{\sigma(k-2)}i_{\sigma(k-1)}}p_{i_{\sigma(k)}i_{k+1}}$$

As an illustration, the basic projective spaces are all monomial:

\begin{proposition}
The spaces $P^{N-1}_\mathbb R\subset P^{N-1}_\mathbb C\subset P^{N-1}_+$ are all monomial, with
$$
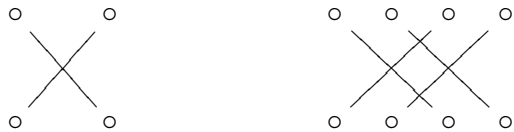
\xymatrix@R=10mm@C=8mm{\circ\ar@{-}[dr]&\circ\ar@{-}[dl]\\\circ&\circ}\qquad\qquad\qquad
\xymatrix@R=10mm@C=3mm{\circ\ar@{-}[drr]&\circ\ar@{-}[drr]&\circ\ar@{-}[dll]&\circ\ar@{-}[dll]\\\circ&\circ&\circ&\circ}$$
producing respectively $P^{N-1}_\mathbb R,P^{N-1}_\mathbb C$.
\end{proposition}

\begin{proof}
We must divide the algebra $C(P^{N-1}_+)$ by the relations associated to the diagrams in the statement, as well as those associated to their shifted versions, given by:
$$
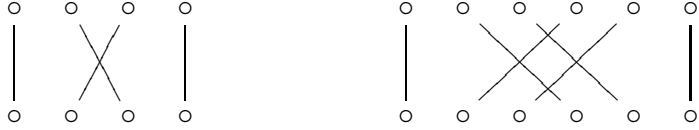
\xymatrix@R=10mm@C=3mm{\circ\ar@{-}[d]&\circ\ar@{-}[dr]&\circ\ar@{-}[dl]&\circ\ar@{-}[d]\\\circ&\circ&\circ&\circ}\qquad\qquad\qquad 
\xymatrix@R=10mm@C=3mm{\circ\ar@{-}[d]&\circ\ar@{-}[drr]&\circ\ar@{-}[drr]&\circ\ar@{-}[dll]&\circ\ar@{-}[dll]&\circ\ar@{-}[d]\\\circ&\circ&\circ&\circ&\circ&\circ}$$ 

(1) The basic crossing, and its shifted version, produce the relations $p_{ab}=p_{ba}$ and $p_{ab}p_{cd}=p_{ac}p_{bd}$. Now by using these relations several times, we obtain:
$$p_{ab}p_{cd}=p_{ac}p_{bd}=p_{ca}p_{db}=p_{cd}p_{ab}$$

Thus, the space produced by the basic crossing is classical, $P\subset P^{N-1}_\mathbb C$, and by using one more time the relations $p_{ab}=p_{ba}$ we conclude that we have $P=P^{N-1}_\mathbb R$.

(2) The fattened crossing, and its shifted version, produce the relations $p_{ab}p_{cd}=p_{cd}p_{ab}$ and $p_{ab}p_{cd}p_{ef}=p_{ad}p_{eb}p_{cf}$. The first relations tell us that the projective space must be classical, $P\subset P^{N-1}_\mathbb C$. Now observe that with $p_{ij}=z_i\bar{z}_j$, the second relations read:
$$z_a\bar{z}_bz_c\bar{z}_dz_e\bar{z}_f=z_a\bar{z}_dz_e\bar{z}_bz_c\bar{z}_f$$

Since these relations are automatic, we have $P=P^{N-1}_\mathbb C$, and we are done.
\end{proof}

We can now formulate our projective classification result, as follows:

\begin{theorem}
The basic $3$ projective spaces, namely 
$$P^{N-1}_\mathbb R\subset P^{N-1}_\mathbb C\subset P^{N-1}_+$$
are the only monomial ones, in the above sense.
\end{theorem}

\begin{proof}
We follow the proof from the affine case. Let $\mathcal R_\sigma$ be the collection of relations associated to a permutation $\sigma\in S_k$ with $k\in 2\mathbb N$, as in Definition 2.9. 

We fix a monomial projective space $P\subset P^{N-1}_+$, and we associate to it a family of subsets $G_k\subset S_k$, as follows:
$$G_k=\begin{cases}
\{\sigma\in S_k|\mathcal R_\sigma\ {\rm hold\ over\ }P\}&(k\ {\rm even})\\
\{\sigma\in S_k|\mathcal R_{|\sigma}\ {\rm hold\ over\ }P\}&(k\ {\rm odd})
\end{cases}$$

As in the affine case, we obtain in this way a filtered group $G=(G_k)$, which is stable under removing outer strings, and under removing neighboring strings. 

Thus the computations in the proof of Theorem 2.8 apply, and show that we have only 3 possible situations, corresponding to the 3 spaces in Proposition 2.10.
\end{proof}

We will see later on, in section 4 below, that the quantum isometry groups of the 3 spheres and 3 projective spaces can be axiomatized as well, in a similar manner.

\section{Quantum isometry groups}

We discuss now the quantum isometry groups of the 3 noncommutative spheres, and of the 3 projective spaces as well. We use the compact quantum group formalism developed by Woronowicz in \cite{wo1}, \cite{wo2}, under the Kac algebra assumption:

\begin{definition}
A finitely generated Hopf $C^*$-algebra is a $C^*$-algebra $A$, given with a unitary matrix $u\in M_N(A)$ whose coefficients generate $A$, such that the formulae
$$\Delta(u_{ij})=\sum_ku_{ik}\otimes u_{kj}\quad,\quad
 \varepsilon(u_{ij})=\delta_{ij}\quad,\quad
S(u_{ij})=u_{ji}^*$$
define morphisms of $C^*$-algebras $\Delta:A\to A\otimes A$, $\varepsilon:A\to\mathbb C$, $S:A\to A^{opp}$.
\end{definition}

The morphisms $\Delta,\varepsilon,S$ are called comultiplication, counit and antipode. Observe that, once the pair $(A,u)$ is given, these morphisms can exist or not. If they exist, they are unique, and we say that we have a finitely generated Hopf $C^*$-algebra.

We have the following result, making the link with standard Hopf algebra theory: 

\begin{proposition}
Let $(A,u)$ be a finitely generated Hopf $C^*$-algebra.
\begin{enumerate} 
\item $\Delta,\varepsilon$ satisfy the usual axioms for a comultiplication and a counit, namely:
\begin{eqnarray*}
(\Delta\otimes id)\Delta&=&(id\otimes \Delta)\Delta\\
(\varepsilon\otimes id)\Delta&=&(id\otimes\varepsilon)\Delta=id
\end{eqnarray*}

\item $S$ satisfies the antipode axiom, on the $*$-subalgebra generated by entries of $u$: 
$$m(S\otimes id)\Delta=m(id\otimes S)\Delta=\varepsilon(.)1$$

\item In addition, the square of the antipode is the identity, $S^2=id$.
\end{enumerate}
\end{proposition}

\begin{proof}
By linearity, involutivity, mutiplicativity and continuity, it is enough to do all the verifications on the coefficients of $u$. The two comultiplication axioms follow from:
\begin{eqnarray*}
(\Delta\otimes id)\Delta(u_{ij})&=&(id\otimes \Delta)\Delta(u_{ij})=\sum_{kl}u_{ik}\otimes u_{kl}\otimes u_{lj}\\
(\varepsilon\otimes id)\Delta(u_{ij})&=&(id\otimes\varepsilon)\Delta(u_{ij})=u_{ij}
\end{eqnarray*}

The two antipode axioms follows from the unitarity of $u$, as follows:
\begin{eqnarray*}
m(S\otimes id)\Delta(u_{ij})&=&\sum_ku_{ki}^*u_{kj}=(u^*u)_{ij}=\delta_{ij}\\
m(id\otimes S)\Delta(u_{ij})&=&\sum_ku_{ik}u_{jk}^*=(uu^*)_{ij}=\delta_{ij}
\end{eqnarray*}

Finally, the extra antipode axiom $S^2=id$ is clear from definitions.
\end{proof}

We say that $A$ is cocommutative when $\Sigma\Delta=\Delta$, where $\Sigma(a\otimes b)=b\otimes a$ is the flip. We have the following result, which justifies the terminology and axioms:

\begin{proposition}
The following are finitely generated Hopf $C^*$-algebras:
\begin{enumerate}
\item $C(G)$, with $G\subset U_N$ compact Lie group. Here the structural maps are:
\begin{eqnarray*}
\Delta(\varphi)&=&(g,h)\to \varphi(gh)\\
\varepsilon(\varphi)&=&\varphi(1)\\
S(\varphi)&=&g\to\varphi(g^{-1})
\end{eqnarray*}

\item $C^*(\Gamma)$, with $F_N\to\Gamma$ finitely generated group. Here the structural maps are:
\begin{eqnarray*}
\Delta(g)&=&g\otimes g\\
\varepsilon(g)&=&1\\ 
S(g)&=&g^{-1}
\end{eqnarray*}

\end{enumerate}
Moreover, we obtain in this way all the commutative/cocommutative algebras.
\end{proposition}

\begin{proof}
In both cases, we have to exhibit a certain matrix $u$. For the first assertion, we can use the matrix $u=(u_{ij})$ formed by matrix coordinates of $G$, given by:
$$g=\begin{pmatrix}
u_{11}(g)&&u_{1N}(g)\\
&\ddots&\\
u_{N1}(g)&&u_{NN}(g)
\end{pmatrix}$$

For the second assertion, we can use the diagonal matrix formed by generators:
$$u=\begin{pmatrix}
g_1&&0\\
&\ddots&\\
0&&g_N
\end{pmatrix}$$

Finally, the last assertion follows from the Gelfand theorem, in the commutative case, and in the cocommutative case, this follows from the results of Woronowicz in \cite{wo1}.
\end{proof}

Observe that the reduced group algebra $C^*_{red}(\Gamma)$ is a finitely generated Hopf $C^*$-algebra as well, with the same defining matrix $u$, and with $\Delta,\varepsilon,S$ given by the same formulae. In order to overcome this issue, we call $A$ full when it is the enveloping $C^*$-algebra of the $*$-algebra generated by the coefficients of $u$. With this notion in hand, we have:

\begin{definition}
Given a finitely generated full Hopf $C^*$-algebra $A$, we write
$$A=C(G)=C^*(\Gamma)$$
and call $G$ compact matrix quantum group, and $\Gamma$ finitely generated quantum group.
\end{definition}

It follows from Proposition 3.3 that when $A$ is both commutative and cocommutative, $G$ is a compact abelian group, $\Gamma$ is a discrete abelian group, and these two groups are in Pontrjagin duality. With suitable terminology and notations, this Pontrjagin type duality extends to the general case, and we write $G=\widehat{\Gamma},\Gamma=\widehat{G}$. See \cite{ntu}, \cite{wo1}.

Let us discuss now the liberation question for the group $O_N$. First, we have:

\begin{proposition}
We have the presentation result
$$C(O_N)=C^*_{comm}\left((u_{ij})_{i,j=1,\ldots,N}\Big|u=\bar{u},u^t=u^{-1}\right)$$
where $C^*_{comm}$ stands as usual for universal commutative $C^*$-algebra.
\end{proposition}

\begin{proof}
This follows from the Gelfand theorem, because the conditions $u=\bar{u},u^t=u^{-1}$ show that the spectrum of the algebra on the right is contained in $O_N$. Thus we have a morphism from left to right, inverse to the trivial morphism from right to left.
\end{proof}

We can now proceed with liberation, in the same way as we did for the spheres:

\begin{proposition}
The following universal $C^*$-algebras
\begin{eqnarray*}
C(O_N^+)&=&C^*\left((u_{ij})_{i,j=1,\ldots,N}\Big|u=\bar{u},u^t=u^{-1}\right)\\
C(O_N^*)&=&C(O_N^+)\Big/\Big<abc=cba,\forall a,b,c\in\{u_{ij}\}\Big>
\end{eqnarray*}
are finitely generated Hopf $C^*$-algebras, and we have $O_N\subset O_N^*\subset O_N^+$.
\end{proposition}

\begin{proof}
In order to prove the first assertion, consider the following three matrices, having coefficients in the target algebras of the maps $\Delta,\varepsilon,S$ to be constructed:
$$u^\Delta_{ij}=\sum_ku_{ik}\otimes u_{kj}\quad,\quad
u^\varepsilon_{ij}=\delta_{ij}\quad,\quad
u^S_{ij}=u_{ji}^*$$

These matrices are all three orthogonal, so the structural maps $\varphi=\Delta,\varepsilon,S$ for the algebra $C(O_N^+)$ can be defined by universality, by setting $\varphi(u_{ij})=u^\varphi_{ij}$.

Regarding now the quotient $C(O_N^*)$, here we know that the entries of $u$ half-commute, and it follows that the entries of $u^\Delta,u^\varepsilon,u^S$ half-commute as well. Thus, once again, we can define the structural maps $\varphi=\Delta,\varepsilon,S$ simply by setting $\varphi(u_{ij})=u^\varphi_{ij}$.
\end{proof}

Summarizing, we have constructed quantum group analogues $O_N\subset O_N^*\subset O_N^+$ of the noncommutative spheres $S^{N-1}_\mathbb R\subset S^{N-1}_{\mathbb R,*}\subset S^{N-1}_{\mathbb R,+}$. Our task now will be to find quantum group analogues of the various results in sections 1-2. We will do this gradually.

We recall from \cite{bve} that the diagonal subgroup of $(G,u)$ is the group dual $\widehat{\Gamma}\subset G$ obtained by setting $C^*(\Gamma)=C(G)/<u_{ij}=0,\forall i\neq j>$, with the remark that this algebra is indeed cocommutative. We have the following result, related to Proposition 1.7:

\begin{proposition}
The diagonal subgroups of the $3$ orthogonal quantum groups are:
$$\begin{matrix}
O_N&\subset&O_N^*&\subset&O_N^+\\
\\
\cup&&\cup&&\cup\\
\\
\widehat{\mathbb Z_2^N}&\subset&\widehat{\mathbb Z_2^{\circ N}}&\subset&\widehat{\mathbb Z_2^{*N}}
\end{matrix}$$
That is, we obtain in this way the noncommutative cubes of $S^{N-1}_\mathbb R\subset S^{N-1}_{\mathbb R,*}\subset S^{N-1}_{\mathbb R,+}$.
\end{proposition}

\begin{proof}
In order to do the computation for $O_N^+$, we must take the universal algebra $C(O_N^+)$ from Proposition 3.6, and divide by the relations $u_{ij}=0$, for $i\neq j$. We obtain:
$$C^*(\Gamma)=C^*\left((u_{ii})_{i=1,\ldots,N}\Big|u_{ii}=u_{ii}^*=u_{ii}^{-1}\right)$$

Thus we have $\Gamma=\mathbb Z_2^{*N}$, and then by taking the quotient by the relations $abc=cba$ and $ab=ba$ we obtain respectively the groups $\mathbb Z_2^{\circ N},\mathbb Z_2^N$, as claimed.
\end{proof}

Regarding now the projective versions, we use here the following notion:

\begin{definition}
Given a closed subgroup $G\subset O_N^+$, its projective version $G\to PG$ is given by the fact that $C(PG)\subset C(G)$ is the subalgebra generated by $w_{ij,ab}=u_{ia}u_{jb}$.
\end{definition}

Here the fact that $PG$ is indeed a compact quantum group comes from the fact that the matrix $w=(w_{ia,jb})$ is a corepresentation. As a basic example, in the classical case, $G\subset O_N$, we obtain in this way the usual projective version, $PG=G/(G\cap\mathbb Z_2^N)$.

We have the following result, coming from \cite{bve}, \cite{bdu}: 

\begin{theorem}
We have an embedding of $C^*$-algebras
$$C(O_N^*)\subset M_2(C(U_N))\quad:\quad u_{ij}\to\begin{pmatrix}0&v_{ij}\\ \bar{v}_{ij}&0\end{pmatrix}$$
where $v_{ij}$ are the standard coordinates on $U_N$. Also, we have $PO_N^*=PU_N$.
\end{theorem}

\begin{proof}
The fact that we have a morphism as above is clear from definitions, and the equality $PO_N^*=PU_N$ can be deduced from this, as in the sphere case. See \cite{bve}.

Now with this equality in hand, the crossed product techniques explained in section 1 apply to our situation, and show that the morphism is faithful. See \cite{bdu}.
\end{proof}

Our goal now will be to prove that $O_N\subset O_N^*\subset O_N^+$ and their projective versions are the quantum isometry groups of the 3 spheres and 3 projective spaces.

We use the following action formalism, inspired from \cite{gos}, \cite{wan}:

\begin{definition}
Consider a closed subgroup $G\subset O_N^+$, and a closed subset $X\subset S^{N-1}_{\mathbb R,+}$.
\begin{enumerate}
\item We write $G\curvearrowright X$ when the formula $\Phi(x_i)=\sum_au_{ia}\otimes x_a$ defines a morphism of $C^*$-algebras $\Phi:C(X)\to C(G)\otimes C(X)$.

\item We write $PG\curvearrowright PX$ when the formula $\Phi(x_ix_j)=\sum_{ab}u_{ia}u_{jb}\otimes x_ax_b$ defines a morphism of $C^*$-algebras $\Phi:C(PX)\to C(PG)\otimes C(PX)$.
\end{enumerate}
\end{definition}

As a first remark, in the case where the above morphisms $\Phi$ exist, they are automatically coaction maps, in the sense that they satisfy the following conditions:
$$(id\otimes\Phi)\Phi=(\Delta\otimes id)\Phi\quad,\quad (\varepsilon\otimes id)\Phi=id$$

We call a closed subset $X\subset S^{N-1}_{\mathbb R,+}$ algebraic when the quotient map $C(S^{N-1}_{\mathbb R,+})\to C(X)$ comes from a collection of polynomial relations between the standard coordinates.
 
We have the following result, from \cite{ba2}:

\begin{proposition}
Assuming that $X\subset S^{N-1}_{\mathbb R,+}$ is algebraic, there exist:
\begin{enumerate}
\item A biggest quantum group $G\subset O_N^+$ acting affinely on $X$.

\item A biggest quantum group $G\subset O_N^+$ acting projectively on $PX$.
\end{enumerate} 
\end{proposition}

\begin{proof}
Assume indeed that there are noncommutative polynomials $P_\alpha$ such that:
$$C(X)=C(S^{N-1}_{\mathbb R,+})/<P_\alpha(x_1,\ldots,x_N)=0>$$

If we want a construct an action $G\curvearrowright X$, the elements $\Phi(x_i)=\sum_au_{ia}\otimes x_a$ must satisfy the relations satisfied by $x_1,\ldots,x_N$. Thus, the universal quantum group $G\subset O_N^+$ as in (1) appears as follows, where $X_i=\sum_au_{ia}\otimes x_a\in C(O_N^+)\otimes C(X)$:
$$C(G)=C(O_N^+)/<P_\alpha(X_1,\ldots,X_N)=0>$$

The proof of (2) is similar, by using the variables $X_{ij}=\sum_{ab}u_{ia}u_{jb}\otimes x_ax_b$. 
\end{proof}

We have the following result, with respect to the above notions:

\begin{theorem}
The quantum isometry groups of the spheres and projective spaces are
$$\begin{matrix}
O_N&\subset&O_N^*&\subset&O_N^+\\
\\
\downarrow&&\downarrow&&\downarrow\\
\\
PO_N&\subset&PU_N&\subset&PO_N^+
\end{matrix}$$
with respect to the affine and projective action notions introduced above.
\end{theorem}

\begin{proof}
The fact that the 3 quantum groups on top act affinely on the corresponding 3 spheres is known since \cite{bgo}, and is elementary. By restriction, the 3 quantum groups on the bottom have actions on the corresponding 3 projective spaces.

We must prove now that all these actions are universal. At right there is nothing to prove, so we are left with studying the actions on $S^{N-1}_\mathbb R,S^{N-1}_{\mathbb R,*}$ and on $P^{N-1}_\mathbb R,P^{N-1}_\mathbb C$.

\underline{$S^{N-1}_\mathbb R$.} Here the fact that the action $O_N\curvearrowright S^{N-1}_\mathbb R$ is universal is known from \cite{bhg}, and follows as well from the fact that the action $PO_N\curvearrowright P^{N-1}_\mathbb R$ is universal, proved below.

\underline{$S^{N-1}_{\mathbb R,*}$.} The situation is similar here, with the universality of $O_N^*\curvearrowright S^{N-1}_{\mathbb R,*}$ being known since \cite{bgo}, and following as well from the universality of $PU_N\curvearrowright P^{N-1}_\mathbb C$, proved below.

\underline{$P^{N-1}_\mathbb R$.} In terms of the projective coordinates $w_{ij,ab}=u_{ia}u_{jb}$ and $p_{ij}=x_ix_j$, the coaction map is given by $\Phi(p_{ij})=\sum_{ab}w_{ij,ab}\otimes p_{ab}$, and we have:
\begin{eqnarray*}
\Phi(p_{ij})&=&\sum_{a<b}(w_{ij,ab}+w_{ij,ba})\otimes p_{ab}+\sum_aw_{ij,aa}\otimes p_{aa}\\
\Phi(p_{ji})&=&\sum_{a<b}(w_{ji,ab}+w_{ji,ba})\otimes p_{ab}+\sum_aw_{ji,aa}\otimes p_{aa}
\end{eqnarray*}

By comparing these two formulae, and then by using the linear independence of the variables $p_{ab}=x_ax_b$ for $a\leq b$, we conclude that we must have:
$$w_{ij,ab}+w_{ij,ba}=w_{ji,ab}+w_{ji,ba}$$

Let us apply now the antipode to this formula. For this purpose, observe first that we have $S(w_{ij,ab})=S(u_{ia}u_{jb})=S(u_{jb})S(u_{ia})=u_{bj}u_{ai}=w_{ba,ji}$. Thus by applying the antipode we obtain $w_{ba,ji}+w_{ab,ji}=w_{ba,ij}+w_{ab,ij}$, and by relabelling, we obtain: 
$$w_{ji,ba}+w_{ij,ba}=w_{ji,ab}+w_{ij,ab}$$

Now by comparing with the original relation, we obtain $w_{ij,ab}=w_{ji,ba}$. But, with $w_{ij,ab}=u_{ia}u_{jb}$, this formula reads $u_{ia}u_{jb}=u_{jb}u_{ia}$. Thus our quantum group $G\subset O_N^+$ must be classical, $G\subset O_N$, and so we have $PG\subset PO_N$, as claimed.

\underline{$P^{N-1}_\mathbb C$.} Consider a coaction map, written as $\Phi(p_{ij})=\sum_{ab}u_{ia}u_{jb}\otimes p_{ab}$, with $p_{ab}=z_a\bar{z}_b$. The idea here will be that of using the formula $p_{ab}p_{cd}=p_{ad}p_{cb}$. We have:
\begin{eqnarray*}
\Phi(p_{ij}p_{kl})&=&\sum_{abcd}u_{ia}u_{jb}u_{kc}u_{ld}\otimes p_{ab}p_{cd}\\
\Phi(p_{il}p_{kj})&=&\sum_{abcd}u_{ia}u_{ld}u_{kc}u_{jb}\otimes p_{ad}p_{cb}
\end{eqnarray*}

The terms at left being equal, and the last terms at right being equal too, we deduce that, with $[a,b,c]=abc-cba$, we must have the following formula:
$$\sum_{abcd}u_{ia}[u_{jb},u_{kc},u_{ld}]\otimes p_{ab}p_{cd}=0$$

Now since the quantities $p_{ab}p_{cd}=z_a\bar{z}_bz_c\bar{z}_d$ at right depend only on the numbers $|\{a,c\}|,|\{b,d\}|\in\{1,2\}$, and this dependence produces the only possible linear relations between the variables $p_{ab}p_{cd}$, we are led to $2\times2=4$ equations, as follows:

(1) $u_{ia}[u_{jb},u_{ka},u_{lb}]=0$, $\forall a,b$.

(2) $u_{ia}[u_{jb},u_{ka},u_{ld}]+u_{ia}[u_{jd},u_{ka},u_{lb}]=0$, $\forall a$, $\forall b\neq d$.

(3) $u_{ia}[u_{jb},u_{kc},u_{lb}]+u_{ic}[u_{jb},u_{ka},u_{lb}]=0$, $\forall a\neq c$, $\forall b$.

(4) $u_{ia}[u_{jb},u_{kc},u_{ld}]+u_{ia}[u_{jd},u_{kc},u_{lb}]+u_{ic}[u_{jb},u_{ka},u_{ld}]+u_{ic}[u_{jd},u_{ka},u_{lb}]=0$, $\forall a\neq c$, $\forall b\neq d$.

We will need in fact only the first two formulae. Since (1) corresponds to (2) at $b=d$, we conclude that (1,2) are equivalent to (2), with no restriction on the indices. By multiplying now this formula to the left by $u_{ia}$, and then summing over $i$, we obtain:
$$[u_{jb},u_{ka},u_{ld}]+[u_{jd},u_{ka},u_{lb}]=0$$

By applying the antipode we get $[u_{dl},u_{ak},u_{bj}]+[u_{bl},u_{ak},u_{dj}]=0$, and by relabelling:
$$[u_{ld},u_{ka},u_{jb}]+[u_{jd},u_{ka},u_{lb}]=0$$

Now by comparing with the original relation, we obtain $[u_{jb},u_{ka},u_{ld}]=[u_{jd},u_{ka},u_{lb}]=0$. Thus our quantum group is half-classical, $G\subset O_N^*$, and we are done.
\end{proof}

\section{Representation theory}

In this section we discuss the quantum group analogues of the axiomatization results obtained in section 2. We use representation theory methods, from  \cite{wo1}, \cite{wo2}.

Let $G$ be a compact matrix quantum group, with defining matrix $u=(u_{ij})$. We let $C^\infty(G)\subset C(G)$ be the dense $*$-subalgebra generated by the coefficients of $u$. Observe that in the classical case this is indeed the algebra of smooth functions on $G$.

Our basic object of study will be the corepresentations of $C(G)$:

\begin{definition}
A finite dimensional smooth unitary corepresentation of $C(G)$ is a unitary matrix $U\in M_K(C^\infty(G))$, satisfying $\Delta(U_{ij})=\sum_kU_{ik}\otimes U_{kj}$ and $\varepsilon(U_{ij})=\delta_{ij}$.
\end{definition}

When $G$ is classical, the corepresentations of $C(G)$ correspond to the representations of $G$. Also, when $G=\widehat{\Gamma}$ is a group dual, the group elements $g\in\Gamma$ are 1-dimensional corepresentations of $C^*(\Gamma)$, and, as explained in \cite{wo1}, each corepresentation of $C^*(\Gamma)$ decomposes as a direct sum of such 1-dimensional corepresentations.

In general now, the defining matrix $u=(u_{ij})$ is a corepresentation, called the fundamental one. By tensoring this corepresentation with itself, we obtain a whole family of corepresentations of $C(G)$, indexed by the positive integers $k\in\mathbb N$:
$$u^{\otimes k}=(u_{i_1j_1}\ldots u_{i_kj_k})_{i_1\ldots i_k,j_1\ldots j_k}$$

Observe that by tensoring $u$ with its complex conjugate we can obtain a bigger family of corepresentations, indexed by the words $k\in\mathbb N*\mathbb N$. However, since our quantum groups have a self-conjugate fundamental representation, we won't need this extension.

We will be interested in computing Schur-Weyl categories, defined as follows:

\begin{definition}
Given $G\subset O_N^+$, the collection $C=(C_{kl})$ of the linear spaces 
$$C_{kl}=Hom(u^{\otimes k},u^{\otimes l})$$
where $u$ is the fundamental corepresentaton, is called the Schur-Weyl category of $G$.
\end{definition}

The Tannakian duality results found by Woronowicz in \cite{wo2} show that this category completely determines $(G,u)$. To be more precise, given a collection of linear spaces $C_{kl}\subset\mathcal L((\mathbb C^N)^{\otimes k},(\mathbb C^N)^{\otimes l})$ which form a tensor $C^*$-category with duals, the associated compact matrix quantum group $(G,u)$ can be constructed as follows:
$$C(G)=C^*\left((u_{ij})_{i,j=1,\ldots,N}\Big|T\in Hom(u^{\otimes k},u^{\otimes l}),\forall k,l\in\mathbb N,\forall T\in C_{kl}\right)$$

Let us go back now to our quantum groups, $O_N\subset O_N^*\subset O_N^+$. It is known since Brauer \cite{bra} that the Schur-Weyl category for $O_N$ is spanned by certain linear maps ``coming from pairings'', and this fundamental result will be our guiding fact. 

Let us denote by $P_2(k,l)$ the set of pairings between an upper row of $k$ points, and a lower row of $l$ points. Observe that $P_2(k,l)=\emptyset$ for $k+l$ odd. We have:

\begin{definition} 
Associated to $\pi\in P_2(k,l)$ is the linear map $T_\pi:(\mathbb C^N)^{\otimes k}\to(\mathbb C^N)^{\otimes l}$,
$$T_\pi(e_{i_1}\otimes\ldots\otimes e_{i_k})=\sum_{j_1\ldots j_l}\delta_\pi\begin{pmatrix}i_1&\ldots&i_k\\ j_1&\ldots&j_l\end{pmatrix}e_{j_1}\otimes\ldots\otimes e_{j_l}$$
where $\delta_\pi\in\{0,1\}$ is the Kronecker type symbol associated to $\pi$.
\end{definition}

To be more precise, in order to construct $\delta_\pi(^i_j)\in\{0,1\}$, let us put the indices of $i=(i_1,\ldots,i_k)$ on the upper $k$ points of $\pi$, and the indices of $j=(j_1,\ldots,j_l)$ on the lower $l$ points of $\pi$. If there is at least one string of $\pi$ joining distinct indices, we set $\delta_\pi(^i_j)=0$. Otherwise, when all strings of $\pi$ join pairs of equal indices, we set $\delta_\pi(^i_j)=1$.

Here are a few basic examples of such maps, of interest in what follows:
$$T_{||\ldots|}=id\quad,\quad T_\cap(1)=\sum_ie_i\otimes e_i\quad,\quad T_\cup(e_i\otimes e_j)=\delta_{ij}$$
$$T_{\slash\!\!\!\backslash}(e_i\otimes e_j)=e_j\otimes e_i\quad,\quad T_{\slash\hskip-1.6mm\backslash\hskip-1.1mm|\hskip0.5mm}(e_i\otimes e_j\otimes e_k)=e_k\otimes e_j\otimes e_i$$

Let us first prove that the usual categorical operations on the linear maps $T_\pi$, namely the composition, tensor product and conjugation, are compatible with the usual categorical operations on the pairings, namely the composition $(\pi,\sigma)\to[^\sigma_\pi]$, the horizontal concatenation $(\pi,\sigma)\to[\pi\sigma]$, and the upside-down turning $\pi\to\pi^*$. We have:

\begin{proposition}
The assignement $\pi\to T_\pi$ is categorical, in the sense that
$$T_\pi\otimes T_\sigma=T_{[\pi\sigma]}\quad,\quad T_\pi T_\sigma=N^{c(\pi,\sigma)}T_{[^\sigma_\pi]}\quad,\quad T_\pi^*=T_{\pi^*}$$
where $c(\pi,\sigma)$ is the number of closed loops obtained when composing.
\end{proposition}

\begin{proof}
The concatenation axiom follows from the following computation:
\begin{eqnarray*}
&&(T_\pi\otimes T_\sigma)(e_{i_1}\otimes\ldots\otimes e_{i_p}\otimes e_{k_1}\otimes\ldots\otimes e_{k_r})\\
&=&\sum_{j_1\ldots j_q}\sum_{l_1\ldots l_s}\delta_\pi\begin{pmatrix}i_1&\ldots&i_p\\j_1&\ldots&j_q\end{pmatrix}\delta_\sigma\begin{pmatrix}k_1&\ldots&k_r\\l_1&\ldots&l_s\end{pmatrix}e_{j_1}\otimes\ldots\otimes e_{j_q}\otimes e_{l_1}\otimes\ldots\otimes e_{l_s}\\
&=&\sum_{j_1\ldots j_q}\sum_{l_1\ldots l_s}\delta_{[\pi\sigma]}\begin{pmatrix}i_1&\ldots&i_p&k_1&\ldots&k_r\\j_1&\ldots&j_q&l_1&\ldots&l_s\end{pmatrix}e_{j_1}\otimes\ldots\otimes e_{j_q}\otimes e_{l_1}\otimes\ldots\otimes e_{l_s}\\
&=&T_{[\pi\sigma]}(e_{i_1}\otimes\ldots\otimes e_{i_p}\otimes e_{k_1}\otimes\ldots\otimes e_{k_r})
\end{eqnarray*}

The composition axiom follows from the following computation:
\begin{eqnarray*}
&&T_\pi T_\sigma(e_{i_1}\otimes\ldots\otimes e_{i_p})\\
&=&\sum_{j_1\ldots j_q}\delta_\sigma\begin{pmatrix}i_1&\ldots&i_p\\j_1&\ldots&j_q\end{pmatrix}
\sum_{k_1\ldots k_r}\delta_\pi\begin{pmatrix}j_1&\ldots&j_q\\k_1&\ldots&k_r\end{pmatrix}e_{k_1}\otimes\ldots\otimes e_{k_r}\\
&=&\sum_{k_1\ldots k_r}N^{c(\pi,\sigma)}\delta_{[^\sigma_\pi]}\begin{pmatrix}i_1&\ldots&i_p\\k_1&\ldots&k_r\end{pmatrix}e_{k_1}\otimes\ldots\otimes e_{k_r}\\
&=&N^{c(\pi,\sigma)}T_{[^\sigma_\pi]}(e_{i_1}\otimes\ldots\otimes e_{i_p})
\end{eqnarray*}

Finally, the involution axiom follows from the following computation:
\begin{eqnarray*}
&&T_\pi^*(e_{j_1}\otimes\ldots\otimes e_{j_q})\\
&=&\sum_{i_1\ldots i_p}<T_\pi^*(e_{j_1}\otimes\ldots\otimes e_{j_q}),e_{i_1}\otimes\ldots\otimes e_{i_p}>e_{i_1}\otimes\ldots\otimes e_{i_p}\\
&=&\sum_{i_1\ldots i_p}\delta_\pi\begin{pmatrix}i_1&\ldots&i_p\\ j_1&\ldots& j_q\end{pmatrix}e_{i_1}\otimes\ldots\otimes e_{i_p}\\
&=&T_{\pi^*}(e_{j_1}\otimes\ldots\otimes e_{j_q})
\end{eqnarray*}

Summarizing, we have proved that our correspondence is indeed categorical.
\end{proof}

We will need as well the following result:

\begin{proposition}
For a closed subgroup $G\subset O_N^+$, the following hold:
\begin{enumerate}
\item $T_{\slash\!\!\!\backslash}\in End(u^{\otimes 2})$ precisely when $G\subset O_N$.

\item $T_{\slash\hskip-1.6mm\backslash\hskip-1.1mm|\hskip0.5mm}\in End(u^{\otimes 3})$ precisely when $G\subset O_N^*$.
\end{enumerate}
\end{proposition}

\begin{proof}
We use the formulae of $T_{\slash\!\!\!\backslash},T_{\slash\hskip-1.6mm\backslash\hskip-1.1mm|\hskip0.5mm}$ given after Definition 4.3 above.

(1) By using $T_{\slash\!\!\!\backslash}(e_i\otimes e_j)=e_j\otimes e_i$, we have the following formulae:
\begin{eqnarray*}
(T_{\slash\!\!\!\backslash}\otimes1)u^{\otimes 2}(e_i\otimes e_j\otimes1)&=&\sum_{kl}e_l\otimes e_k\otimes u_{ki}u_{lj}\\
u^{\otimes 2}(\bar{T}_{\slash\!\!\!\backslash}\otimes1)(e_i\otimes e_j\otimes1)&=&\sum_{kl}e_l\otimes e_k\otimes u_{li}u_{ki}
\end{eqnarray*}

We therefore obtain the commutation relations $ab=ba$, and we are done.

(2) By using $T_{\slash\hskip-1.6mm\backslash\hskip-1.1mm|\hskip0.5mm}(e_i\otimes e_j\otimes e_k)=e_k\otimes e_j\otimes e_i$, we have the following formulae:
\begin{eqnarray*}
(T_{\slash\hskip-1.6mm\backslash\hskip-1.1mm|\hskip0.5mm}\otimes1)u^{\otimes 2}(e_i\otimes e_j\otimes e_k\otimes1)
&=&\sum_{abc}e_c\otimes e_b\otimes e_a\otimes u_{ai}u_{bj}u_{ck}\\
u^{\otimes 2}(\bar{T}_{\slash\hskip-1.6mm\backslash\hskip-1.1mm|\hskip0.5mm}\otimes1)(e_i\otimes e_j\otimes e_k\otimes1)
&=&\sum_{abc}e_c\otimes e_b\otimes e_a\otimes u_{ck}u_{bj}u_{ai}
\end{eqnarray*}

We therefore obtain the half-commutation relations $abc=cba$, and we are done.
\end{proof}

We can now formulate the Brauer theorem, as well as its noncommutative generalizations, regarding $O_N^*,O_N^+$. Let us call a pairing ``balanced'' if, when labelling cyclically its legs $\bullet\circ\bullet\circ\ldots$, each string connects a black leg to a while leg. We have:

\begin{theorem}
The spaces $C_{kl}$ for the quantum groups $O_N\subset O_N^*\subset O_N^+$ are respectively
$$span(P_2(k,l))\supset span(P_2^*(k,l))\supset span(NC_2(k,l))$$
where $P_2\supset P_2^*\supset NC_2$ are the pairings, balanced pairings, and noncrossing pairings.
\end{theorem}

\begin{proof}
This follows indeed from Proposition 4.4, Proposition 4.5, and from the Tannakian duality results established by Woronowicz in \cite{wo2}. Indeed, by Proposition 4.4 each of the 3 collections of spaces in the statement is a tensor $C^*$-category, which must therefore correspond to a certain quantum group $G\subset O_N^+$. But Proposition 4.5 shows that these quantum groups are exactly those in the statement, and we are done.
\end{proof}

As a first application, we will prove that, under suitable axioms, $O_N^*$ is the only intermediate object $O_N\subset G\subset O_N^+$, and $PU_N$ is the only intermediate object $PO_N\subset G\subset PO_N^+$. In order to formulate our statement, we recall the following notion, from \cite{bsp}:

\begin{definition}
An intermediate quantum group $O_N\subset G\subset O_N^+$ is called easy when 
$$span(NC_2(k,l))\subset Hom(u^{\otimes k},u^{\otimes l})\subset span(P_2(k,l))$$
comes via $Hom(u^{\otimes k},u^{\otimes l})=span(D(k,l))$, for certain sets of pairings $D(k,l)$.
\end{definition}

Observe that $O_N,O_N^*,O_N^+$ are all easy, due to Theorem 4.6 above.

In general now, in the context of the above definition, observe that by ``saturating'' the sets $D(k,l)$, we can always assume that the collection $D=(D(k,l))$ is a category of pairings, in the sense that it is stable under the vertical and horizontal concatenation, and the upside-down turning, and contains the semicircle. See \cite{bsp}.

With the above notion in hand, we have the following result, from \cite{bve}:

\begin{theorem}
The only intermediate easy quantum groups 
$$O_N\subset G\subset O_N^+$$ 
are the basic orthogonal quantum groups, $O_N\subset O_N^*\subset O_N^+$.
\end{theorem}

\begin{proof}
We agree that the points of a pairing $\pi\in P_2(k,l)$ are counted counterclockwise starting from bottom left, and modulo $k+l$. For $i=1, 2, \ldots,k+l$ we denote by $\pi^i$ the partition obtained by connecting with a semicircle the $i$-th and $(i+1)$-th points. The partitions $\pi^i$ will be called ``cappings'' of $\pi$, and will be generically denoted $\pi'$. 

\underline{Step I.} Let $\pi\in P_2-NC_2$, having $s\geq 4$ strings. Our claim is that:
\begin{enumerate}
\item If $\pi\in P_2-P_2^*$, there exists a capping $\pi'\in P_2-P_2^*$.

\item If $\pi\in P_2^*-NC_2$, there exists a capping $\pi'\in P_2^*-NC_2$.
\end{enumerate}

Indeed, we can use a rotation in order to assume that $\pi$ has no upper points. In other words, our data is a partition $\pi\in P_2(0, 2s)-NC_2(0, 2s)$, with $s\geq 4$.

(1) The assumption $\pi\notin P_2^*$ tells us that $\pi$ has certain strings having an odd number of crossings. We fix such an ``odd'' string, and we try to cap $\pi$, as for this string to remain odd in the resulting partition $\pi'$. An examination of all the possible pictures shows that this is possible, provided that our partition has $s\geq 3$ strings, and we are done.

(2) The assumption $\pi\notin NC_2$ tells us that $\pi$ has certain crossing strings. We fix such a pair of crossing strings, and we try to cap $\pi$, as for these strings to remain crossing in $\pi'$. Once again, an examination of all the possible pictures shows that this is possible, provided that our partition has $s\geq 4$ strings, and we are done.

\underline{Step II.} Consider a partition $\pi\in P_2(k,l)-NC_2(k,l)$. Our claim is that:
\begin{enumerate}
\item If $\pi\in P_2(k, l)-P_2^*(k,l)$ then $<\pi>=P_2$.

\item If $\pi\in P_2^*(k,l)-NC_2(k,l)$ then $<\pi>=P_2^*$.
\end{enumerate}

This can be proved by recurrence on the number of strings, $s=(k+l)/2$. Indeed, by using the results in Step I, at any $s\geq 4$ we have a descent procedure $s\to s-1$, and this leads to the situation $s\in\{1,2,3\}$, where the statement is clear.

\underline{Step III.} Assume now that we are given an easy quantum group $O_N\subset G\subset O_N^+$, coming from certain sets of pairings $D(k,l)\subset P_2(k,l)$. We have three cases:

(1) If $D\not\subset P_2^*$, we obtain $G=O_N$.

(2) If $D\subset P_2,D\not\subset NC_2$, we obtain $G=O_N^*$.

(3) If $D\subset NC_2$, we obtain $G=O_N^+$.
\end{proof}

In the projective case now, we have the following notion, from \cite{bme}:

\begin{definition}
A projective category of pairings is a collection of subsets 
$$NC_2(2k,2l)\subset E(k,l)\subset P_2(2k,2l)$$
stable under the usual categorical operations, and satisfying $\sigma\in E\implies |\sigma|\in E$.
\end{definition}

As basic examples here, we have the categories $NC_2\subset P_2^*\subset P_2$, where $P_2^*$ is the category of balanced pairings. This follows indeed from definitions.

Now with the above notion in hand, we can formulate:

\begin{definition}
A quantum group $PO_N\subset H\subset PO_N^+$ is called projectively easy when 
$$span(NC_2(2k,2l))\subset Hom(v^{\otimes k},v^{\otimes l})\subset span(P_2(2k,2l))$$
comes via $Hom(v^{\otimes k},v^{\otimes l})=span(E(k,l))$, for a certain projective category $E=(E(k,l))$.
\end{definition}

Observe that, given any easy quantum group $O_N\subset G\subset O_N^+$, its projective version $PO_N\subset PG\subset PO_N^+$ is projectively easy in our sense. In particular the quantum groups $PO_N\subset PU_N\subset PO_N^+$ are all projectively easy, coming from $NC_2\subset P_2^*\subset P_2$.

We have in fact the following general result:

\begin{proposition}
We have a bijective correspondence between the affine and projective categories of partitions, given by $G\to PG$ at the quantum group level.
\end{proposition}

\begin{proof}
The construction of correspondence $D\to E$ is clear, simply by setting:
$$E(k,l)=D(2k,2l)$$

Conversely, given $E=(E(k,l))$ as in Definition 4.10, we can set:
$$D(k,l)=\begin{cases}
E(k,l)&(k,l\ {\rm even})\\
\{\sigma:|\sigma\in E(k+1,l+1)\}&(k,l\ {\rm odd})
\end{cases}$$

Our claim is that $D=(D(k,l))$ is a category of partitions. Indeed:

(1) The composition action is clear. Indeed, when looking at the numbers of legs involved, in the even case this is clear, and in the odd case, this follows from:
$$|\sigma,|\tau\in E\implies |^\sigma_\tau\in E\implies{\ }^\sigma_\tau\in D$$

(2) For the tensor product axiom, we have 4 cases to be investigated. The even/even case is clear, and the odd/even, even/odd, odd/odd cases follow respectively from:
$$|\sigma,\tau\in E\implies|\sigma\tau\in E\implies\sigma\tau\in D$$
$$\sigma,|\tau\in E\implies|\sigma|,|\tau\in E\implies|\sigma||\tau\in E\implies|\sigma\tau\in E\implies\sigma\tau\in D$$
$$|\sigma,|\tau\in E\implies||\sigma|,|\tau\in E\implies||\sigma||\tau\in E\implies \sigma\tau\in E\implies\sigma\tau\in D$$

(3) Finally, the conjugation axiom is clear from definitions.

Now with these definitions in hand, both compositions $D\to E\to D$ and $E\to D\to E$ follow to be the identities, and the quantum group assertion is clear as well.
\end{proof}

Now back to the uniqueness issues, we have here:

\begin{theorem}
The only intermediate projectively easy quantum groups 
$$PO_N\subset G\subset PO_N^+$$ 
are the basic projective quantum groups, $PO_N\subset PU_N\subset PO_N^+$.
\end{theorem}

\begin{proof}
This follows from the uniqueness result in the affine case, Theorem 4.8 above, and from the duality established in Proposition 4.11.
\end{proof}

We refer to \cite{rwe} for further results regarding the easy quantum groups.

\section{The Weingarten formula}

In this section we develop the integration theory over the quantum groups $O_N^\times$, and then over the associated spheres $S^{N-1}_{\mathbb R,\times}$, by using ideas from \cite{bco}, \cite{bgo}, \cite{csn}, \cite{wei}. 

Assume first that $G$ is an arbitrary compact quantum group. We have:

\begin{definition}
The unique positive unital linear form $\int:C(G)\to\mathbb C$ satisfying
$$\left(id\otimes\int\right)\Delta(\varphi)=\left(\int\otimes id\right)\Delta(\varphi)=\int\varphi$$
for any $\varphi\in C(G)$ is called Haar functional of $G$.
\end{definition}

The existence of the Haar functional can be established by starting with an arbitrary positive unital linear form $f:C(G)\to\mathbb C$, and then by performing convolution powers, with the convolution product being given by $f*g=(f\otimes g)\Delta$:
$$\int=\lim_{k\to\infty}\frac{1}{k}\sum_{r=1}^kf^{*r}$$

Let us also mention that, due to our restricted axioms, assuming $S^2=id$, the Haar functional is a trace. We refer to Woronowicz' paper \cite{wo1} for the proof of these facts.

At the level of basic examples, the situation is as follows:

\begin{proposition}
The Haar functional is as follows:
\begin{enumerate}
\item In the classical case, $\int$ is the integration with respect to the Haar measure.

\item In the group dual case, $G=\widehat{\Gamma}$, the integration is given by $\int g=\delta_{g1}$, $\forall g\in\Gamma$.
\end{enumerate}
\end{proposition}

\begin{proof}
These assertions are both elementary, as follows:

(1) When $G$ is classical we must have $\int\varphi=\int_G\varphi(g)d\mu(g)$, for a certain probability measure $\mu$. The conditions in Definition 5.1 express the fact that $\mu$ must be left and right invariant, $\mu(gX)=\mu(Xg)=\mu(X)$. Thus $\mu$ must be the Haar measure of $G$.

(2) When $G=\widehat{\Gamma}$ is a group dual, the group elements $g\in\Gamma$ span the dense subalgebra $\mathbb C[\Gamma]\subset C^*(\Gamma)$, and the invariance conditions in Definition 5.1 applied to them simply read $g\int g=g\int g=\int g$. Thus the Haar functional is given by $\int g=\delta_{g1}$ in this case.
\end{proof}

In practice, $\int$ can be computed by using representation theory, by using:

\begin{proposition}
For any corepresentation $U\in M_K(C(G))$, the operator 
$$P=\left(id\otimes\int\right)U\in M_K(\mathbb C)$$
is the orthogonal projection onto the space $Fix(U)=\{x\in\mathbb C^K|U(x)=x\otimes1\}$.
\end{proposition}

\begin{proof}
The invariance conditions in Definition 5.1 applied to $\varphi=U_{ij}$ read:
$$\sum_kU_{ik}P_{kj}=\sum_kP_{ik}U_{kj}=P_{ij}$$

Thus we have $UP=PU=P$, and this gives the result. See Woronowicz \cite{wo1}.
\end{proof}

Now back to our orthogonal quantum groups, the Schur-Weyl duality results obtained in section 4 above, along with a linear algebra trick, give the following result:

\begin{theorem}
We have the Weingarten formula
$$\int_{O_N^\times}u_{i_1j_1}\ldots u_{i_kj_k}=\sum_{\pi,\sigma\in P_2^\times(k)}\delta_\pi(i)\delta_\sigma(j)W_{kN}(\pi,\sigma)$$
where $W_{kN}=G_{kN}^{-1}$, with $G_{kN}(\pi,\sigma)=N^{|\pi\vee\sigma|}$, and where $\delta$ are Kronecker type symbols.
\end{theorem}

\begin{proof}
In view of the above, let us arrange all the integrals to be computed, at a fixed value of $k\in\mathbb N$, in a single big matrix, of size $N^k\times N^k$, as follows:
$$P_{i_1\ldots i_k,j_1\ldots j_k}=\int_{O_N^\times}u_{i_1j_1}\ldots u_{i_kj_k}$$

By using Proposition 5.3, and then Theorem 4.6, this matrix $P$ is the orthogonal projection onto the following linear space:
$$Fix(u^{\otimes k})=Hom(1,u^{\otimes k})=span\left(\xi_\pi\Big|\pi\in P_2^\times(k)\right)$$ 

By a standard linear algebra computation, it follows that we have $P=WE$, where $E(x)=\sum_{\pi\in P_2^\times(k)}<x,\xi_\pi>\xi_\pi$, and where $W$ is the inverse on $span(T_\pi|\pi\in P_2^\times(k))$ of the restriction of $E$. But this restriction is the linear map given by $G_{kN}$, so $W$ is the linear map given by $W_{kN}$, and this gives the formula in the statement.
\end{proof}

We recall now that the uniform measure on $S^{N-1}_\mathbb R$ is the unique probability measure which is invariant by rotations. In the case of the half-liberated sphere $S^{N-1}_{\mathbb R,*}$ and of the free sphere $S^{N-1}_{\mathbb R,+}$, establishing such a uniqueness result is a quite non-trivial task, and we will have to use the Weingarten formula. Let us begin with:

\begin{definition}
We endow $C(S^{N-1}_{\mathbb R,\times})$ with its canonical trace, coming as the composition
$$tr:C(S^{N-1}_{\mathbb R,\times})\to C(O_N^\times)\to\mathbb C$$
of the morphism given by $x_i\to u_{1i}$, and of the Haar integral of $O_N^\times$.
\end{definition}

Observe that the morphism $C(S^{N-1}_{\mathbb R,\times})\to C(O_N^\times)$ is indeed well-defined, because the variables $X_i=u_{1i}$ satisfy the defining relations for the coordinates of $S^{N-1}_{\mathbb R,\times}$.

In the classical case, we obtain in this way the integration with respect to the uniform measure on $S^{N-1}_\mathbb R$. This is well-known, and follows as well from the results below.

In general now, let us first find an abstract characterization of this canonical trace, via invariance, similar to the one from the classical case. We will need:

\begin{proposition}
The canonical trace $tr:C(S^{N-1}_{\mathbb R,\times})\to\mathbb C$ has the ergodicity property
$$(I\otimes id)\Phi=tr(.)1$$
where $I$ is the Haar functional of $O_N^\times$, and $\Phi$ is the coaction map.
\end{proposition}

\begin{proof}
It is enough to check the equality in the statement on an arbitrary product of coordinates, $x_{i_1}\ldots x_{i_k}$. The left term is as follows:
\begin{eqnarray*}
(I\otimes id)\Phi(x_{i_1}\ldots x_{i_k})
&=&\sum_{j_1\ldots j_k}I(u_{i_1j_1}\ldots u_{i_kj_k})x_{j_1}\ldots x_{j_k}\\
&=&\sum_{j_1\ldots j_k}\sum_{\pi,\sigma\in P_2^\times(k)}\delta_\pi(i)\delta_\sigma(j)W_{kN}(\pi,\sigma)x_{j_1}\ldots x_{j_k}\\
&=&\sum_{\pi,\sigma\in P_2^\times(k)}\delta_\pi(i)W_{kN}(\pi,\sigma)\sum_{j_1\ldots j_k}\delta_\sigma(j)x_{j_1}\ldots x_{j_k}
\end{eqnarray*}

Let us look now at the last sum on the right. The situation is as follows:

(1) In the free case we have to sum quantities of type $x_{j_1}\ldots x_{j_k}$, over all choices of multi-indices $j=(j_1,\ldots,j_k)$ which fit into our given noncrossing pairing $\sigma$, and just by using the condition $\sum_ix_i^2=1$, we conclude that the sum is 1. 

(2) The same happens in the classical case. Indeed, our pairing $\sigma$ can now be crossing, but we can use the commutation relations $x_ix_j=x_jx_i$, and the sum is again 1.

(3) Finally, the same happens as well in the half-liberated case, because the fact that our pairing $\sigma$ has now an even number of crossings allows us to use the half-commutation relations $x_ix_jx_k=x_kx_jx_i$, in order to conclude that the sum to be computed is 1. 

Thus the sum on the right is 1, in all cases, and we get:
$$(I\otimes id)\Phi(x_{i_1}\ldots x_{i_k})
=\sum_{\pi,\sigma\in P_2^\times(k)}\delta_\pi(i)W_{kN}(\pi,\sigma)$$

On the other hand, another application of the Weingarten formula gives:
\begin{eqnarray*}
tr(x_{i_1}\ldots x_{i_k})
&=&I(u_{1i_1}\ldots u_{1i_k})\\
&=&\sum_{\pi,\sigma\in P_2^\times(k)}\delta_\pi(1)\delta_\sigma(i)W_{kN}(\pi,\sigma)\\
&=&\sum_{\pi,\sigma\in P_2^\times(k)}\delta_\sigma(i)W_{kN}(\pi,\sigma)
\end{eqnarray*}

Since the Weingarten function is symmetric in $\pi,\sigma$, this gives the result.
\end{proof} 

With the above ergodicity result in hand, we can now formulate an abstract characterization of the trace constructed in Definition 5.5, as follows:

\begin{theorem}
There is a unique positive unital trace $tr:C(S^{N-1}_{\mathbb R,\times})\to\mathbb C$ satisfying
$$(id\otimes tr)\Phi(x)=tr(x)1$$
where $\Phi$ is the coaction map, and this is the canonical trace.
\end{theorem}

\begin{proof}
First of all, it follows from the invariance condition in Definition 5.1, applied to $G=O_N^\times$, that the canonical trace has the invariance property in the statement.

In order to prove now the uniqueness, let $\tau$ be as in the statement. We have:
$$\tau (I\otimes id)\Phi(x)=(I\otimes\tau)\Phi(x)=I(id\otimes\tau)\Phi(x)=I(\tau(x)1)=\tau(x)$$

On the other hand, according to Proposition 5.6, we have as well:
$$\tau (I\otimes id)\Phi(x)=\tau(tr(x)1)=tr(x)$$

Thus we obtain $\tau=tr$, and this finishes the proof.
\end{proof}

We discuss in what follows the various properties of $tr$. As a first observation, in practice we can always do our computations over $O_N^\times$, because we have:

\begin{proposition}
The following algebras, with generators and traces, are isomorphic, when replaced with their GNS completions with respect to their canonical traces:
\begin{enumerate}
\item The algebra $C(S^{N-1}_{\mathbb R,\times})$, with generators $x_1,\ldots,x_N$.

\item The row algebra $R_N^\times\subset C(O_N^\times)$ generated by $u_{11},\ldots,u_{1N}$.
\end{enumerate}
\end{proposition}

\begin{proof}
Consider the quotient map $\pi:C(S^{N-1}_{\mathbb R,\times})\to R_N^\times$, which was used for constructing $tr$. The invariance property of the integration functional $I:C(O_N^\times)\to\mathbb C$ shows that $tr'=I\pi$ satisfies the invariance condition in Theorem 5.7, so we have $tr=tr'$. Together with the positivity of $tr$ and with the basic properties of the GNS construction, this shows that $\pi$ induces an isomorphism at the level of GNS algebras, as claimed.
\end{proof}

We make now the following convention, for the reminder of this section: 

\begin{definition}
We agree from now on to replace each algebra $C(S^{N-1}_{\mathbb R,\times})$ with its GNS completion with respect to the canonical trace.
\end{definition}

As a first observation, the classical sphere $S^{N-1}_\mathbb R$ is left unchanged by this modification, because the trace comes from the usual uniform measure on it. The same happens for the half-liberated sphere $S^{N-1}_{\mathbb R,*}$, due to the injectivity of the morphism in Theorem 1.14 above. The free sphere $S^{N-1}_{\mathbb R,+}$, however, is probably ``cut'' by this construction, for instance because this happens at the quantum group level, where $O_N^+$ is not coamenable.

With the above convention in hand, we can now discuss the geometry of our spheres. Our main result here will be the fact that, with a suitable formalism, the universal affine actions constructed in section 3 above are ``isometric''. Let us begin with:

\begin{definition}
Given a compact Riemannian manifold $X$, we denote by $\Omega^1(X)$ the space of smooth $1$-forms on $X$, with scalar product given by
$$<\omega,\eta>=\int_X<\omega(x),\eta(x)>dx$$
and we construct the Hodge Laplacian $\Delta:L^2(X)\to L^2(X)$ by setting $\Delta=d^*d$, where $d:C^\infty(X)\to\Omega^1(X)$ is the usual differential map, and $d^*$ is its adjoint. 
\end{definition}

According to a standard result in differential geometry, the isometry group $\mathcal G(X)$ is then the group of diffeomorphisms $\varphi:X\to X$ whose induced action on $C^\infty(X)$ commutes with $\Delta$. Following now Goswami \cite{gos}, we can formulate:

\begin{definition}
Associated to a compact Riemannian manifold $X$ are:
\begin{enumerate}
\item ${\rm Diff}^+(X)$: the category of compact quantum groups acting on $X$.

\item $\mathcal G^+(X)\in {\rm Diff}^+(X)$: the universal object with a coaction commuting with $\Delta$.
\end{enumerate}
\end{definition}

Here the coactions $\Phi:C(X)\to C(G)\otimes C(X)$ must satisfy by definition the axioms in section 3 above, as well as the following smoothness assumption:
$$\Phi(C^\infty(X))\subset C(G)\otimes C^\infty(X)$$

As for the commutation condition with $\Delta$ in (2) above, this regards the canonical extension of the action to the space $L^2(X)$. For details here, see \cite{gos}.

Let us discuss now the case of the noncommutative Riemannian manifolds. We use here some very light axioms, inspired from Connes' ones from \cite{con}:

\begin{definition}
A compact spectral triple $X=(A,H,D)$ consists of the following:
\begin{enumerate}
\item $A$ is a unital $C^*$-algebra.

\item $H$ is a Hilbert space, on which $A$ acts.

\item $D$ is an unbounded self-adjoint operator on $H$, with compact resolvents, such that $[D,\phi]$ has a bounded extension, for any $\phi$ in a dense $*$-subalgebra $\mathcal A\subset A$.
\end{enumerate}
\end{definition}

The guiding examples come from the compact Riemannian manifolds $X$. Indeed, associated to $X$ are several triples $(A,H,D)$, with $\mathcal A\subset A$ being $C^\infty(X)\subset C(X)$:
\begin{enumerate}
\item $H$ is the space of square-integrable spinors, and $D$ is the Dirac operator.

\item $H$ is the space of forms on $X$, and $D$ is the Hodge-Dirac operator $d+d^*$.

\item $H=L^2(X,dv)$, $dv$ being the Riemannian volume, and $D=d^*d$. 
\end{enumerate}

Here the first example is the most interesting one, because under a number of supplementary axioms, a reconstruction theorem for $X$ holds, in terms of $(A,H,D)$. See \cite{con}. In view of Definition 5.11 (2), however, the last example will be in fact the one that we will be interested in. Once again following Goswami \cite{gos}, we have:

\begin{definition}
Associated to a compact spectral triple $X=(A,H,D)$ are:
\begin{enumerate}
\item ${\rm Diff}^+(X)$: the category of compact quantum groups acting on $(A,H)$.

\item $\mathcal G^+(X)\in {\rm Diff}^+(X)$: the universal object with a coaction commuting with $D$.
\end{enumerate}
\end{definition}

In other words, $\mathcal G^+(X)$ must have a unitary representation $U$ on $H$ which commutes with $D$, satisfies $U1_A=1\otimes1_A$, and is such that $ad_U$ maps $A''$ into itself.

Now back to our spheres, we will construct now a spectral triple, satisfying the conditions in Definition 5.12. The idea is to use the inclusion $S^{N-1}_\mathbb R\subset S^{N-1}_{\mathbb R,\times}$, and to construct the Laplacian filtration as the pullback of the Laplacian filtration for $S^{N-1}_\mathbb R$.

More precisely, we have the following construction:

\begin{proposition}
Associated to $S^{N-1}_{\mathbb R,\times}$ is the triple $(A,H,D)$, where $A=C(S^{N-1}_{\mathbb R,\times})$, and where $D$ acting on $H=L^2(A,tr)$ is defined as follows:
\begin{enumerate}
\item Set $H_k=span(x_{i_1}\ldots x_{i_r}|r\leq k)$.

\item Define $E_k=H_k\cap H_{k-1}^\perp$, so that we have $H=\oplus_{k=0}^\infty E_k$.

\item Finally, set $Dx=\lambda_kx$, for any $x\in E_k$, where $\lambda_k$ are distinct numbers.
\end{enumerate} 
\end{proposition}

\begin{proof}
We have to show that $[D,T_i]$ is bounded, where $T_i$ is the left multiplication by $x_i$. Since $x_i\in A$ is self-adjoint, so is the corresponding operator $T_i$. Now since $T_i(H_k)\subset H_{k+1}$, by self-adjointness we get $T_i(H_k^\perp)\subset H_{k-1}^\perp$. Thus we have: 
$$T_i(E_k)\subset E_{k-1}\oplus E_k\oplus E_{k+1}$$

This gives a decomposition of type $T_i=T_i^{-1}+T_i^0+T_i^1$. It is routine to check that we have $[D,T^\alpha_i]=\alpha T^\alpha_i$ for any $\alpha\in\{-1,0,1\}$, and this gives the result.
\end{proof}

As an illustration, in the classical case the situation is as follows:

\begin{proposition}
For the sphere $S^{N-1}_\mathbb R$ we have $D=f(\sqrt{d^*d})$, where: 
$$f(s)=1-\frac{N}{2}+\frac{1}{2}\sqrt{4s^2+(N-2)^2}$$
In particular, the eigenspaces of $D$ and $\sqrt{d^*d}$ coincide. 
\end{proposition}

\begin{proof}
This follows from the well-known fact that $\sqrt{d^*d}$ diagonalizes as in Proposition 5.14, the corresponding eigenvalues being $k(k+N-2)$, with $k\in\mathbb N$.
\end{proof}

In general, it is quite unclear what the eigenvalues should be. This question, raised some time ago in \cite{bgo}, is still open. However, with our formalism, we can now prove:

\begin{theorem}
We have $\mathcal G^+(S^{N-1}_{\mathbb R,\times})=O_N^\times$.
\end{theorem}

\begin{proof}
Consider the coaction map $\Phi:C(S^{N-1}_{\mathbb R,\times})\to C(O_N^\times)\otimes C(S^{N-1}_{\mathbb R,\times})$. This extends to a unitary representation on the GNS space $H$, that we denote by $U$. We have $\Phi(H_k)\subset C(O_N^\times)\otimes H_k$, which reads $U(H_k)\subset H_k$. By unitarity we get as well $U(H_k^\perp)\subset H_k^\perp$, so each $E_k$ is $U$-invariant, and $U,D$ must commute. We conclude that $\Phi$ is isometric with respect to $D$. Finally, the universality of $O_N^\times$ follows from Theorem 3.12.
\end{proof}

\section{Hyperspherical laws}

In this section we investigate the problem of computing the integral over the spheres $S^{N-1}_\mathbb R\subset S^{N-1}_{\mathbb R,*}\subset S^{N-1}_{\mathbb R,+}$ of polynomial quantities of type $x_{i_1}\ldots x_{i_k}$. 

As a first observation, we have the following elementary result:

\begin{proposition}
We have the formula
$$\int_{S^{N-1}_{\mathbb R,\times}}x_{i_1}\ldots x_{i_k}\,dx=0$$
unless each $x_i$ appears an even number of times.
\end{proposition}

\begin{proof}
This follows from the fact that for any $i$ we have an automorphism of $C(S^{N-1}_{\mathbb R,\times})$ given by $x_i\to -x_i$. Indeed, this automorphism must preserve the trace, so if $x_i$ appears an odd number of times, the integral in the statement satisfies $I=-I$, so $I=0$.
\end{proof}

In order to compute the remaining integrals, which are non-trivial, we can use the Weingarten formula, adapted to the sphere case. The statement here is:

\begin{proposition}
We have the Weingarten type formula
$$\int_{S^{N-1}_{\mathbb R,\times}}x_{i_1}\ldots x_{i_k}\,dx=\sum_{\sigma\in P_2^\times(k)}\delta_\sigma(i)\sum_{\pi\in P_2^\times(k)}W_{kN}(\pi,\sigma)$$
where $W_{kN}=G_{kN}^{-1}$, with $G_{kN}(\pi,\sigma)=N^{|\pi\vee\sigma|}$, and where $\delta$ are Kronecker type symbols.
\end{proposition}

\begin{proof}
This follows indeed from the Weingarten formula for the quantum groups, via the identification $x_i=u_{1i}$ coming from Proposition 5.8 above.
\end{proof}

As an illustration, we can recover in this way the vanishing result in Proposition 6.1. Indeed, if some $x_i$ appears an odd number of times, we have $\delta_\sigma(i)=0$ for any $\sigma$.

As an application now, we will compute the $N\to\infty$ behavior of the hyperspherical laws. For this purpose, we will need some standard facts from free probability theory \cite{bpa}, \cite{nsp}, \cite{vdn}. Assume that $(A,tr)$ is a $C^*$-algebra with a trace. We have:

\begin{definition}
The law of $a=a^*\in A$ is the real probability measure given by:
$$tr(a^k)=\int_\mathbb Rx^kd\mu(x)$$
Equivalently, $\mu$ comes from the restriction $tr:C(X)\subset A\to\mathbb C$, where $X=Spec(a)$.
\end{definition}

We recall that the real Gaussian distribution, appearing in classical probability theory via the central limit theorem (CLT) has as density $\frac{1}{\sqrt{2\pi}}e^{-x^2/2}dx$. 

In Voiculescu's free probability theory, the CLT procedure makes sense as well, and converges to the semicircle law on $[-2,2]$, having density $\frac{1}{2\pi}\sqrt{4-x^2}dx$.

Finally, let us recall that the moduli of the complex Gaussian variables are called Rayleigh variables. We refer to \cite{bpa}, \cite{nsp}, \cite{vdn} for all this material.
 
Now back to the noncommutative spheres, our result here is as follows:

\begin{theorem}
With $N\to\infty$, the variables $\sqrt{N}x_i\in C(S^{N-1}_{\mathbb R,\times})$ are as follows:
\begin{enumerate}
\item $S^{N-1}_\mathbb R$: real Gaussian.

\item $S^{N-1}_{\mathbb R,*}$: symmetrized Rayleigh.

\item $S^{N-1}_{\mathbb R,+}$: semicircular.
\end{enumerate}
\end{theorem}

\begin{proof}
We use the Weingarten formula, from Proposition 6.2 above. Since with $N\to\infty$ the Gram matrix $G_{kN}(\pi,\sigma)=N^{|\pi\vee\sigma|}$ is asymptotically constant, $G_{kN}(\pi,\sigma)\simeq\delta_{\pi,\sigma}N^{k/2}$, its inverse is asymptotically constant as well, $W_{kN}(\pi,\sigma)\simeq\delta_{\pi,\sigma}N^{-k/2}$, and so:
$$\int_{S^{N-1}_{\mathbb R,\times}}x_{i_1}\ldots x_{i_k}\,dx\simeq N^{-k/2}\sum_{\sigma\in P_2^\times(k)}\delta_\sigma(i)$$

With this formula in hand, we can compute the asymptotic moments of each coordinate $x_i$. Indeed, by setting $i_1=\ldots=i_k=i$, all Kronecker symbols are 1, and we get:
$$\int_{S^{N-1}_{\mathbb R,\times}}x_i^k\,dx\simeq N^{-k/2}\#(P_2^\times(k))$$

Thus the even asymptotic moments of $\sqrt{N}x_i$ are the numbers $\#(P_2^\times(2l))$, which are equal respectively to $(2l)!!,l!,\frac{1}{l+1}\binom{2l}{l}$, and this gives the result.
\end{proof}

Regarding now the ``$N$ fixed'' problematics, we first have the following result:

\begin{proposition}
The spherical integral of $x_{i_1}\ldots x_{i_k}$ vanishes, unless each $a\in\{1,\ldots,N\}$ appears an even number of times in the sequence $i_1,\ldots,i_k$. We have
$$\int_{S^{N-1}_\mathbb R}x_{i_1}\ldots x_{i_k}\,dx=\frac{(N-1)!!l_1!!\ldots l_N!!}{(N+\sum l_i-1)!!}$$
where $m!!=(m-1)(m-1)(m-5)\ldots$, and $l_a$ is this number of occurrences.
\end{proposition}

\begin{proof}
First, the result holds indeed at $N=2$, due to the following well-known formula, where $\varepsilon(p)=1$ when $p\in\mathbb N$ is even, and $\varepsilon(p)=0$ when $p$ is odd:
$$\int_0^{\pi/2}\cos^pt\sin^qt\,dt=\left(\frac{\pi}{2}\right)^{\varepsilon(p)\varepsilon(q)}\frac{p!!q!!}{(p+q+1)!!}$$

In general now, in view of Proposition 6.1, we can restrict attention to the case $l_a\in 2\mathbb N$. The integral in the statement can be written in spherical coordinates, as follows:
$$I=\frac{2^N}{V}\int_0^{\pi/2}\ldots\int_0^{\pi/2}x_1^{l_1}\ldots x_N^{l_N}J\,dt_1\ldots dt_{N-1}$$

Here $V$ is the volume of the sphere, $J$ is the Jacobian, and the $2^N$ factor comes from the restriction to the $1/2^N$ part of the sphere where all the coordinates are positive.

The normalization constant in front of the integral is:
$$\frac{2^N}{V}
=\frac{2^N}{N\pi^{N/2}}\cdot\Gamma\left(\frac{N}{2}+1\right)
=\left(\frac{2}{\pi}\right)^{[N/2]}(N-1)!!$$

As for the unnormalized integral, this is given by:
\begin{eqnarray*}
I'=\int_0^{\pi/2}\ldots\int_0^{\pi/2}
&&(\cos t_1)^{l_1}
(\sin t_1\cos t_2)^{l_2}\\
&&\ldots\ldots\ldots\\
&&(\sin t_1\sin t_2\ldots\sin t_{N-2}\cos t_{N-1})^{l_{N-1}}\\
&&(\sin t_1\sin t_2\ldots\sin t_{N-2}\sin t_{N-1})^{l_N}\\
&&\sin^{N-2}t_1\sin^{N-3}t_2\ldots\sin^2t_{N-3}\sin t_{N-2}\\
&&dt_1\ldots dt_{N-1}
\end{eqnarray*}

By rearranging the terms, we obtain:
\begin{eqnarray*}
I'
&=&\int_0^{\pi/2}\cos^{l_1}t_1\sin^{l_2+\ldots+l_N+N-2}t_1\,dt_1\\
&&\int_0^{\pi/2}\cos^{l_2}t_2\sin^{l_3+\ldots+l_N+N-3}t_2\,dt_2\\
&&\ldots\ldots\ldots\\
&&\int_0^{\pi/2}\cos^{l_{N-2}}t_{N-2}\sin^{l_{N-1}+l_N+1}t_{N-2}\,dt_{N-2}\\
&&\int_0^{\pi/2}\cos^{l_{N-1}}t_{N-1}\sin^{l_N}t_{N-1}\,dt_{N-1}
\end{eqnarray*}

Now by using the above-mentioned formula at $N=2$, this gives:
\begin{eqnarray*}
I'
&=&\frac{l_1!!(l_2+\ldots+l_N+N-2)!!}{(l_1+\ldots+l_N+N-1)!!}\left(\frac{\pi}{2}\right)^{\varepsilon(N-2)}\\
&&\frac{l_2!!(l_3+\ldots+l_N+N-3)!!}{(l_2+\ldots+l_N+N-2)!!}\left(\frac{\pi}{2}\right)^{\varepsilon(N-3)}\\
&&\ldots\ldots\ldots\\
&&\frac{l_{N-2}!!(l_{N-1}+l_N+1)!!}{(l_{N-2}+l_{N-1}+l_N+2)!!}\left(\frac{\pi}{2}\right)^{\varepsilon(1)}\\
&&\frac{l_{N-1}!!l_N!!}{(l_{N-1}+l_N+1)!!}\left(\frac{\pi}{2}\right)^{\varepsilon(0)}
\end{eqnarray*}

Now observe that the various double factorials multiply up to quantity in the statement, modulo a $(N-1)!!$ factor, and that the $\frac{\pi}{2}$ factors multiply up to $(\frac{\pi}{2})^{[N/2]}$. Thus by multiplying with the normalization constant, we obtain the result.
\end{proof}

In the case of the half-liberated sphere, we have the following result:

\begin{proposition}
The half-liberated integral of $x_{i_1}\ldots x_{i_k}$ vanishes, unless each index $a$ appears the same number of times at odd and even positions in $i_1,\ldots,i_k$. We have
$$\int_{S^{N-1}_{\mathbb R,*}}x_{i_1}\ldots x_{i_k}\,dx=4^{\sum l_i}\frac{(2N-1)!l_1!\ldots l_n!}{(2N+\sum l_i-1)!}$$
where $l_a$ denotes this number of common occurrences.
\end{proposition}

\begin{proof}
In view of Proposition 6.1 above, we can assume that $k$ is even, $k=2l$. The corresponding integral can be viewed as an integral over $S^{N-1}_\mathbb C$, as follows:
$$I=\int_{S^{N-1}_\mathbb C}z_{i_1}\bar{z}_{i_2}\ldots z_{i_{2l-1}}\bar{z}_{i_{2l}}\,dz$$

Now by using the same argument as in the proof of Proposition 6.1, but this time with transformations of type $p\to\lambda p$ with $|\lambda|=1$, we see that $I$ vanishes, unless each $z_a$ appears as many times as $\bar{z}_a$ does, and this gives the first assertion.

Assume now that we are in the non-vanishing case. Then the $l_a$ copies of $z_a$ and the $l_a$ copies of $\bar{z}_a$ produce by multiplication a factor $|z_a|^{2l_a}$, so we have:
$$I=\int_{S^{N-1}_\mathbb C}|z_1|^{2l_1}\ldots|z_N|^{2l_N}\,dz$$

Now by using the standard identification $S^{N-1}_\mathbb C\simeq S^{2N-1}_\mathbb R$, we obtain:
\begin{eqnarray*}
I
&=&\int_{S^{2N-1}_\mathbb R}(x_1^2+y_1^2)^{l_1}\ldots(x_N^2+y_N^2)^{l_N}\,d(x,y)\\
&=&\sum_{r_1\ldots r_N}\binom{l_1}{r_1}\ldots\binom{l_N}{r_N}\int_{S^{2N-1}_\mathbb R}x_1^{2l_1-2r_1}y_1^{2r_1}\ldots x_N^{2l_N-2r_N}y_N^{2r_N}\,d(x,y)
\end{eqnarray*}

By using the formula in Proposition 6.5, we obtain:
\begin{eqnarray*}
I
&=&\sum_{r_1\ldots r_N}\binom{l_1}{r_1}\ldots\binom{l_N}{r_N}\frac{(2N-1)!!(2r_1)!!\ldots(2r_N)!!(2l_1-2r_1)!!\ldots (2l_N-2r_N)!!}{(2N+2\sum l_i-1)!!}\\
&=&\sum_{r_1\ldots r_N}\binom{l_1}{r_1}\ldots\binom{l_N}{r_N}
\frac{(2N-1)!(2r_1)!\ldots (2r_N)!(2l_1-2r_1)!\ldots (2l_N-2r_N)!}{(2N+\sum l_i-1)!r_1!\ldots r_N!(l_1-r_1)!\ldots (l_N-r_N)!}
\end{eqnarray*}

We can rewrite the sum on the right in the following way:
\begin{eqnarray*}
I
&=&\sum_{r_1\ldots r_N}\frac{l_1!\ldots l_N!(2N-1)!(2r_1)!\ldots (2r_N)!(2l_1-2r_1)!\ldots (2l_N-2r_N)!}{(2N+\sum l_i-1)!(r_1!\ldots r_N!(l_1-r_1)!\ldots (l_N-r_N)!)^2}\\
&=&\sum_{r_1}\binom{2r_1}{r_1}\binom{2l_1-2r_1}{l_1-r_1}\ldots\sum_{r_N}\binom{2r_N}{r_N}\binom{2l_N-2r_N}{l_N-r_N}\frac{(2N-1)!l_1!\ldots l_N!}{(2N+\sum l_i-1)!}
\end{eqnarray*}

The sums on the right being $4^{l_1},\ldots,4^{l_N}$, this gives the formula in the statement.
\end{proof}

Finally, in the case of the free sphere, we have the following result, from \cite{bcz}:

\begin{theorem}
The moments of the free hyperspherical law are given by
$$\int_{S^{N-1}_{\mathbb R,+}}x_1^{2l}\,dx=\frac{1}{(N+1)^l}\cdot\frac{q+1}{q-1}\cdot\frac{1}{l+1}\sum_{r=-l-1}^{l+1}(-1)^r\begin{pmatrix}2l+2\cr l+r+1\end{pmatrix}\frac{r}{1+q^r}$$
where $q\in [-1,0)$ is such that $q+q^{-1}=-N$.
\end{theorem}

\begin{proof}
The idea is that $x_1\in C(S^{N-1}_{\mathbb R,+})$ has the same law as $u_{11}\in C(O_N^+)$, which has the same law as a certain variable $w\in C(SU^q_2)$, which can be in turn modelled by an explicit operator on $l^2(\mathbb N)$, whose law can be computed by using advanced calculus.

Let us first explain the relation between $O_N^+$ and $SU^q_2$. To any matrix $F\in GL_N(\mathbb R)$ satisfying $F^2=1$ we associate the following universal algebra:
$$C(O_F^+)=C^*\left((u_{ij})_{i,j=1,\ldots,N}\Big|u=F\bar{u}F={\rm unitary}\right)$$

Observe that $O_{I_N}^+=O_N^+$. In general, the above algebra satisfies Woronowicz's generalized axioms in \cite{wo1}, which do not include the strong antipode axiom $S^2=id$.

At $N=2$, up to a trivial equivalence relation on the matrices $F$, and on the quantum groups $O_F^+$, we can assume that $F$ is as follows, with $q\in [-1,0)$:
$$F=\begin{pmatrix}0&\sqrt{-q}\\
1/\sqrt{-q}&0\end{pmatrix}$$

Our claim is that for this matrix we have $O_F^+=SU^q_2$. Indeed, the relations $u=F\bar{u}F$ tell us that $u$ must be of the following special form:
$$u=\begin{pmatrix}\alpha&-q\gamma^*\cr \gamma&\alpha^*\end{pmatrix}$$

Thus $C(O_F^+)$ is the universal algebra generated by two elements $\alpha,\gamma$, with the relations making the above matrix $u$ unitary. But these unitarity conditions are:
$$\alpha\gamma=q\gamma\alpha,\quad 
\alpha\gamma^*=q\gamma^*\alpha,\quad
\gamma\gamma^*=\gamma^*\gamma,\quad
\alpha^*\alpha+\gamma^*\gamma=1,\quad
\alpha\alpha^*+q^2\gamma\gamma^*=1$$

We recognize here the relations in \cite{wo1} defining the algebra $C(SU^q_2)$, and it follows that we have an isomorphism of Hopf $C^*$-algebras $C(O_F^+)\simeq C(SU^q_2)$.

Now back to the general case, let us try to understand the integration over $O_F^+$. Given $\pi\in NC_2(2k)$ and $i=(i_1,\ldots,i_{2k})$, we set $\delta_\pi^F(i)=\prod_{s\in\pi}F_{i_{s_l}i_{s_r}}$, with the product over all the strings $s=\{s_l\curvearrowright s_r\}$ of $\pi$. Our claim is that the following family of vectors, with $\pi\in NC_2(2k)$, spans the space of fixed vectors of $u^{\otimes 2k}$, for the quantum group $O_F^+$:
$$\xi_\pi=\sum_i\delta_\pi^F(i)e_{i_1}\otimes\ldots\otimes e_{i_{2k}}$$ 

Indeed, having $\xi_\cap$ fixed by $u^{\otimes 2}$ is equivalent to assuming that $u=F\bar{u}F$ is unitary. 

By using now the above vectors, we obtain the following Weingarten formula:
$$\int_{O_F^+}u_{i_1j_1}\ldots u_{i_{2k}j_{2k}}=\sum_{\pi\sigma}\delta_\pi^F(i)\delta_\sigma^F(j)W_{kN}(\pi,\sigma)$$

With these preliminaries in hand, let us start the computation. Let $N\in\mathbb N$, and consider the number $q\in [-1,0)$ satisfying $q+q^{-1}=-N$. Our claim is that we have:
$$\int_{O_N^+}\varphi(\sqrt{N+2}\,u_{ij})=\int_{SU^q_2}\varphi(\alpha+\alpha^*+\gamma-q\gamma^*)$$

Indeed, the moments of the variable on the left are given by:
$$\int_{O_N^+}u_{ij}^{2k}=\sum_{\pi\sigma}W_{kN}(\pi,\sigma)$$

On the other hand, the moments of the variable on the right, which in terms of the fundamental corepresentation $v=(v_{ij})$ is given by $w=\sum_{ij}v_{ij}$, are given by:
$$\int_{SU^q_2}w^{2k}=\sum_{ij}\sum_{\pi\sigma}\delta_\pi^F(i)\delta_\sigma^F(j)W_{kN}(\pi,\sigma)$$

We deduce that $w/\sqrt{N+2}$ has the same moments as $u_{ij}$, which proves our claim.

In order to do now the computation over $SU^q_2$, we can use a matrix model due to Woronowicz \cite{wo1}, where the standard generators $\alpha,\gamma$ are mapped as follows:
\begin{eqnarray*}
\pi_u(\alpha)e_k&=&\sqrt{1-q^{2k}}e_{k-1}\\
\pi_u(\gamma)e_k&=&uq^k e_k
\end{eqnarray*}

Here $u\in\mathbb T$ is a parameter, and $(e_k)$ is the standard basis of $l^2(\mathbb N)$. The point with this representation is that it allows the computation of the Haar functional. Indeed, if $D$ is the diagonal operator given by $D(e_k)=q^{2k}e_k$, then the formula is as follows:
$$\int _{SU^q_2}x=(1-q^2)\int_{\mathbb T}tr(D\pi_u(x))\frac{du}{2\pi iu}$$

With the above model in hand, the law of the variable that we are interested in is as follows, where $M(e_k)=e_{k+1}+q^k(u-qu^{-1})e_k+(1-q^{2k})e_{k-1}$:
$$\int_{SU^q_2}\varphi(\alpha+\alpha^*+\gamma-q\gamma^*)=(1-q^2)\int_{\mathbb T}tr(D\varphi(M))\frac{du}{2\pi iu}$$

The point now is that the integral on the right can be computed, by using advanced calculus methods, and this gives the result. We refer here to \cite{bcz}. 
\end{proof}

The computation of the joint free hypersperical laws remains an open problem. Open as well is the question of finding a more conceptual proof for the above formula.

\section{Twisting results}

We have seen in the previous sections that, under very strong axioms, there are only three noncommutative spheres, $S^{N-1}_\mathbb R\subset S^{N-1}_{\mathbb R,*}\subset S^{N-1}_{\mathbb R,+}$. Moreover, these spheres can be sucessfully studied by using their quantum isometry groups, $O_N\subset O_N^*\subset O_N^+$.

We discuss now some extensions of these facts. The idea is that $S^{N-1}_\mathbb R\subset S^{N-1}_{\mathbb R,*}\subset S^{N-1}_{\mathbb R,+}$ should be thought of as corresponding to a Drinfeld-Jimbo parameter $q=1$, and the question is about what happens when using the next simplest parameter, $q=-1$.

Let us begin with the definition of the twisted spheres, from \cite{ba1}:

\begin{definition}
The twisted spheres $\bar{S}^{N-1}_\mathbb R\subset\bar{S}^{N-1}_{\mathbb R,*}\subset S^{N-1}_{\mathbb R,+}$ are constructed by imposing the following conditions on the standard coordinates $x_1,\ldots,x_N$:
\begin{enumerate}
\item $\bar{S}^{N-1}_\mathbb R$: $x_ix_j=-x_jx_i$, for any $i\neq j$.

\item $\bar{S}^{N-1}_{\mathbb R,*}$: $x_ix_jx_k=-x_kx_jx_i$ for $i,j,k$ distinct, $x_ix_jx_k=x_kx_jx_i$ otherwise.
\end{enumerate}
\end{definition}

Here the fact that we have indeed an inclusion $\bar{S}^{N-1}_\mathbb R\subset\bar{S}^{N-1}_{\mathbb R,*}$ comes from the computations $abc=-bac=bca=-cba$ for $a,b,c\in\{x_i\}$ distinct, and $aab=-aba=baa$ for $a,b\in\{x_i\}$ distinct, where $x_1,\ldots,x_N$ are the standard coordinates on $\bar{S}^{N-1}_\mathbb R$.

Let us also mention that $S^{N-1}_{\mathbb R,+}$ cannot be twisted, or rather, that it is equal to its own twist. We will come back later to this issue, with some results in this direction.

Let us discuss now the twisting of $O_N,O_N^*$. The definition here is as follows:

\begin{definition}
The twisted quantum groups $\bar{O}_N\subset\bar{O}_N^*\subset O_N^+$ are constructed by imposing the following conditions on the standard coordinates $u_{ij}$,
$$\bar{O}_N:ab=\begin{cases}
-ba&{\rm for}\ a\neq b\ {\rm on\ the\ same\ row\ or\ column\ of\ }u\\
ba&{\rm otherwise}
\end{cases}$$
$$\hskip-20.9mm \bar{O}_N^*:abc=\begin{cases}
-cba&{\rm for\ }r\leq2,s=3{\rm\ or\ }r=3,s\leq2\\
cba&{\rm for\ }r\leq2,s\leq 2{\rm\ or\ }r=s=3
\end{cases}$$
where $r,s\in\{1,2,3\}$ are the number of rows/columns of $u$ spanned by $a,b,c\in\{u_{ij}\}$.
\end{definition}

It is routine to check that both $\bar{O}_N,\bar{O}_N^*$ are indeed quantum groups, and that we have an inclusion $\bar{O}_N\subset\bar{O}_N^*$. These facts, as well as Definition 7.2 itself, are however best understood from a Schur-Weyl viewpoint. We will develop the Schur-Weyl theory now, and we will come back later to the spheres, with a quantum isometry group result.

Let us first fine-tune our partition formalism, as follows:

\begin{definition}
We let $P(k,l)$ be the set of partitions between an upper row of $k$ points and a lower row of $l$ points, and consider the following subsets of $P(k,l)$:
\begin{enumerate}
\item $P_2(k,l)\subset P_{even}(k,l)$: the pairings, and the partitions with blocks having even size.

\item $NC_2(k,l)\subset NC_{even}(k,l)\subset NC(k,l)$: the subsets of noncrossing partitions. 

\item $Perm(k,k)\subset P_2(k,k)$: the pairings having only up-to-down strings.
\end{enumerate}
\end{definition}

Given $\pi\in P(k,l)$, we can always switch pairs of neighbors, belonging to different blocks, either in the upper row, or in the lower row, as to make $\pi$ noncrossing. We will need the following standard result, regarding the behavior of this operation:

\begin{proposition}
There is a signature map $\varepsilon:P_{even}\to\{-1,1\}$, given by $\varepsilon(\pi)=(-1)^c$, where $c$ is the number of switches needed to make $\pi$ noncrossing. In addition:
\begin{enumerate}
\item For $\pi\in Perm(k,k)\simeq S_k$, this is the usual signature.

\item For $\pi\in P_2$ we have $(-1)^c$, where $c$ is the number of crossings.

\item For $\pi\in P$ obtained from $\sigma\in NC_{even}$ by merging blocks, the signature is $1$.
\end{enumerate}
\end{proposition}

\begin{proof}
We must first prove that the number $c$ in the statement is well-defined modulo 2. In order to do so, observe that any partition $\pi\in P(k,l)$ can be put in ``standard form'', by ordering its blocks according to the appearence of the first leg in each block, counting clockwise from top left, and then by performing the switches as for block 1 to be at left, then for block 2 to be at left, and so on. Here the required switches are also uniquely determined, by the order coming from counting clockwise from top left. 

Here is an example of such an algorithmic switching operation, for a pairing:
$$
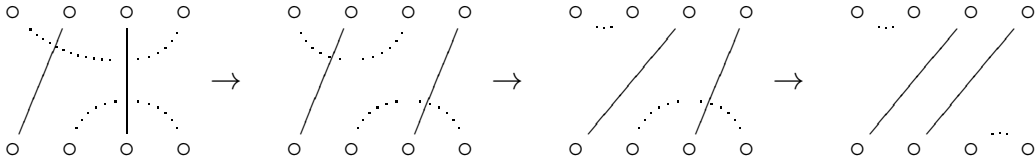
\xymatrix@R=3mm@C=3mm{\circ\ar@/_/@{.}[drr]&\circ\ar@{-}[dddl]&\circ\ar@{-}[ddd]&\circ\\
&&\ar@/_/@{.}[ur]&\\
&&\ar@/^/@{.}[dr]&\\
\circ&\circ\ar@/^/@{.}[ur]&\circ&\circ}
\xymatrix@R=6mm@C=1mm{&\\\to\\&\\& }
\xymatrix@R=3mm@C=3mm{\circ\ar@/_/@{.}[dr]&\circ\ar@{-}[dddl]&\circ&\circ\ar@{-}[dddl]\\
&\ar@/_/@{.}[ur]&&\\
&&\ar@/^/@{.}[dr]&\\
\circ&\circ\ar@/^/@{.}[ur]&\circ&\circ}
\xymatrix@R=6mm@C=1mm{&\\\to\\&\\&}
\xymatrix@R=3mm@C=3mm{\circ\ar@/_/@{.}[r]&\circ&\circ\ar@{-}[dddll]&\circ\ar@{-}[dddl]\\
&&&\\
&&\ar@/^/@{.}[dr]&\\
\circ&\circ\ar@/^/@{.}[ur]&\circ&\circ}
\xymatrix@R=6mm@C=1mm{&\\\to\\&\\& }
\xymatrix@R=3mm@C=3mm{\circ\ar@/_/@{.}[r]&\circ&\circ\ar@{-}[dddll]&\circ\ar@{-}[dddll]\\
&&&\\
&&&\\
\circ&\circ&\circ\ar@/^/@{.}[r]&\circ}$$
\vskip-8mm
 
The point now is that, under the assumption $\pi\in NC_{even}(k,l)$, each of the moves required for putting a leg at left, and hence for putting a whole block at left, requires an even number of switches. Thus, putting $\pi$ is standard form requires an even number of switches. Now given $\pi,\pi'\in P_2$ having the same block structure, the standard form coincides, so the number of switches $c$ required for the passage $\pi\to\pi'$ is indeed even.

Regarding now the remaining assertions, these can be proved as follows:

(1) For $\pi\in Perm(k,k)$ the standard form is $\pi'=id$, and the passage $\pi\to id$ comes by composing with a number of transpositions, which gives the signature. 

(2) For a general $\pi\in P_2$, the standard form is of type $\pi'=|\ldots|^{\cup\ldots\cup}_{\cap\ldots\cap}$, and the passage $\pi\to\pi'$ requires $c$ mod 2 switches, where $c$ is the number of crossings. 

(3) For a partition $\pi\in P_{even}$ coming from $\sigma\in NC_{even}$ by merging a certain number $n$ of blocks, the fact that the signature is 1 follows by recurrence on $n$. 
\end{proof}

We can make act partitions in $P_{even}$ on tensors in a twisted way, as follows:

\begin{definition}
Associated to any partition $\pi\in P_{even}(k,l)$ is the linear map
$$\bar{T}_\pi(e_{i_1}\otimes\ldots\otimes e_{i_k})=\sum_{\sigma\leq\pi}\varepsilon(\sigma)\sum_{j:\ker(^i_j)=\sigma}e_{j_1}\otimes\ldots\otimes e_{j_l}$$
where $\varepsilon:P_{even}\to\{-1,1\}$ is the signature map.
\end{definition}

Here, and in what follows, the kernel of a multi-index is the partition obtained by joining pairs of multi-indices. Observe the similarity with Definition 4.3.

Let us first prove that the above construction is categorical:

\begin{proposition}
The assignement $\pi\to\bar{T}_\pi$ is categorical, in the sense that
$$\bar{T}_\pi\otimes\bar{T}_\sigma=\bar{T}_{[\pi\sigma]}\quad,\quad\bar{T}_\pi\bar{T}_\sigma=N^{c(\pi,\sigma)}\bar{T}_{[^\sigma_\pi]}\quad,\quad\bar{T}_\pi^*=\bar{T}_{\pi^*}$$
where $c(\pi,\sigma)$ is the number of closed loops obtained when composing.
\end{proposition}

\begin{proof}
We follow the proof of Proposition 4.4. We just have to understand the behavior of the twisted version of the Kronecker symbol construction $\pi\to\delta_\pi$, under the various categorical operations on the partitions $\pi$, and the verification goes as follows:

\underline{1. Concatenation.} It is enough to check the following formula:
$$\varepsilon\left(\ker\begin{pmatrix}i_1&\ldots&i_p\\ j_1&\ldots&j_q\end{pmatrix}\right)
\varepsilon\left(\ker\begin{pmatrix}k_1&\ldots&k_r\\ l_1&\ldots&l_s\end{pmatrix}\right)=
\varepsilon\left(\ker\begin{pmatrix}i_1&\ldots&i_p&k_1&\ldots&k_r\\ j_1&\ldots &j_q&l_1&\ldots&l_s\end{pmatrix}\right)$$

Let us denote by $\pi,\sigma$ the partitions on the left, so that the partition on the right is of the form $\rho\leq[\pi\sigma]$. By switching to the noncrossing form, $\pi\to\pi'$ and $\sigma\to\sigma'$, the partition on the right transforms into $\rho\to\rho'\leq[\pi'\sigma']$. Now since $[\pi'\sigma']$ is noncrossing, we can use Proposition 7.4 (3), and we obtain the result.

\underline{2. Composition.} Here we must establish the following formula:
$$\varepsilon\left(\ker\begin{pmatrix}i_1&\ldots&i_p\\ j_1&\ldots&j_q\end{pmatrix}\right)
\varepsilon\left(\ker\begin{pmatrix}j_1&\ldots&j_q\\ k_1&\ldots&k_r\end{pmatrix}\right)
=\varepsilon\left(\ker\begin{pmatrix}i_1&\ldots&i_p\\ k_1&\ldots&k_r\end{pmatrix}\right)$$

Let $\pi,\sigma$ be the partitions on the left, so that the partition on the right is of the form $\rho\leq[^\pi_\sigma]$. Our claim is that we can jointly switch $\pi,\sigma$ to the noncrossing form. Indeed, we can first switch as for $\ker(j_1\ldots j_q)$ to become noncrossing, and then switch the upper legs of $\pi$, and the lower legs of $\sigma$, as for both these partitions to become noncrossing. 

Now observe that when switching in this way to the noncrossing form, $\pi\to\pi'$ and $\sigma\to\sigma'$, the partition on the right transforms into $\rho\to\rho'\leq[^{\pi'}_{\sigma'}]$. Now since $[^{\pi'}_{\sigma'}]$ is noncrossing, we can apply Proposition 7.4 (3), and we obtain the result.

\underline{3. Involution.} Here we must prove the following formula:
$$\varepsilon\left(\ker\begin{pmatrix}i_1&\ldots&i_p\\ j_1&\ldots&j_q\end{pmatrix}\right)=\varepsilon\left(\ker\begin{pmatrix}j_1&\ldots&j_q\\ i_1&\ldots&i_p\end{pmatrix}\right)$$

But this formula is trivial, and this finishes the proof.
\end{proof}

We can now formulate an abstract twisting result, as follows:

\begin{theorem}
The Schur-Weyl categories for $\bar{O}_N,\bar{O}_N^*,O_N^+$ are given by
$$Hom(u^{\otimes k},u^{\otimes l})=span(\bar{T}_\pi|\pi\in D(k,l))$$
for any $k,l\in\mathbb N$, where $D\subset P_2$ is the category of pairings for $O_N,O_N^*,O_N^+$. 
\end{theorem}

\begin{proof}
The correspondence $\pi\to\bar{T}_\pi$ being categorical, the linear spaces in the statement form a tensor $C^*$-category, which produces via \cite{wo2} a compact quantum group $\bar{G}\subset O_N^+$. We must prove that this quantum group is precisely $\bar{G}=\bar{O}_N,\bar{O}_N^*,O_N^+$.

First of all, the result is clear for $O_N^+$, because Proposition 7.4 (1) shows that for any $\pi\in NC_{even}$, and in particular for any $\pi\in NC_2$, we have $T_\pi=\bar{T}_\pi$.

In order to deal now with $O_N,O_N^*$, observe first that we have:
$$\bar{T}_{\slash\!\!\!\backslash}(e_i\otimes e_j)
=\begin{cases}
-e_j\otimes e_i&{\rm for}\ i\neq j\\
e_j\otimes e_i&{\rm otherwise}
\end{cases}$$
$$\ \ \ \ \ \ \ \ \ \bar{T}_{\slash\hskip-1.6mm\backslash\hskip-1.1mm|\hskip0.5mm}(e_i\otimes e_j\otimes e_k)
=\begin{cases}
-e_k\otimes e_j\otimes e_i&{\rm for}\ i,j,k\ {\rm distinct}\\
e_k\otimes e_j\otimes e_i&{\rm otherwise}
\end{cases}$$

Indeed, the basic crossings $\slash\!\!\!\backslash=\ker(^{ab}_{ba}),\slash\hskip-2.0mm\backslash\hskip-1.7mm|=\ker(^{abc}_{cba})$ are both odd, because they have respectively 1 and 3 crossings, and their various subpartitions are as follows, all even:
$$\ker\begin{pmatrix}a&a\\a&a\end{pmatrix},\ \ker\begin{pmatrix}a&a&b\\b&a&a\end{pmatrix},\ \ker\begin{pmatrix}a&b&a\\a&b&a\end{pmatrix},\ \ker\begin{pmatrix}b&a&a\\a&a&b\end{pmatrix},\ \ker\begin{pmatrix}a&a&a\\a&a&a\end{pmatrix}$$

Now since the relations $\bar{T}_{\slash\!\!\!\backslash}\in End(u^{\otimes 2})$, $\bar{T}_{\slash\hskip-1.6mm\backslash\hskip-1.1mm|\hskip0.5mm}\in End(u^{\otimes 3})$ correspond precisely to the relations in Definition 7.2, defining $\bar{O}_N,\bar{O}_N^*$, this gives the result.
\end{proof}

As an application, we can now integrate over the twisted quantum groups:

\begin{theorem}
We have the Weingarten formula
$$\int_{\bar{O}_N^\times}u_{i_1j_1}\ldots u_{i_kj_k}=\sum_{\pi,\sigma\in P_2^\times(k)}\bar{\delta}_\pi(i_1,\ldots,i_k)\bar{\delta}_\sigma(j_1,\ldots, j_k)W_{kN}(\pi,\sigma)$$
where $\bar{\delta}_\pi(i)\in\{-1,0,1\}$ is equal to $\varepsilon(\ker(i))$ if $\ker(i)\leq\pi$, and is $0$ otherwise.
\end{theorem}

\begin{proof}
We know from Theorem 7.7 that $Fix(u^{\otimes k})$ is spanned by the following vectors:
$$\bar{\xi}_\pi=\sum_{j_1\ldots j_k}\bar{\delta}_\pi(j_1,\ldots,j_k)e_{j_1}\otimes\ldots\otimes e_{j_k}$$

The result follows then as in the untwisted case, with the remark that we have:
\begin{eqnarray*}
<\bar{\xi}_\pi,\bar{\xi}_\sigma>
&=&\left\langle\sum_{j:\ker j\leq\pi}\varepsilon(\ker j)e_{j_1}\otimes\ldots\otimes e_{j_l},\sum_{j:\ker j\leq\sigma}\varepsilon(\ker j)e_{j_1}\otimes\ldots\otimes e_{j_l}\right\rangle\\
&=&\sum_{j:\ker j\leq(\pi\vee\sigma)}\varepsilon(\ker j)^2=\sum_{j:\ker j\leq(\pi\vee\sigma)}1=N^{|\pi\vee\sigma|}
\end{eqnarray*}

Thus the Weingarten matrix is indeed the same as in the classical case.
\end{proof}

With these results in hand, let us go back now to the twisted spheres. In order to compute the quantum isometry groups, we will need the following technical result:

\begin{proposition}
The following elements are linearly independent:
\begin{enumerate}
\item $\{x_ax_b|a\leq b\}$, over $\bar{S}^{N-1}_\mathbb R$.

\item $\{x_ax_bx_c|a\leq c\}$ over $\bar{S}^{N-1}_{\mathbb R,*}$.
\end{enumerate}
\end{proposition}

\begin{proof}
We use the morphism $C(\bar{S}^{N-1}_{\mathbb R,\times})\to C(\bar{O}_N^\times)$, given by $x_i\to u_{1i}$. Thus, it is enough to prove the corresponding statements over $\bar{O}_N^\times$, with $x_i=u_{1i}$.

(1) The scalar products between the variables in the statement are:
$$<x_ax_b,x_ix_j>=\int_{\bar{O}_N}u_{1a}u_{1b}u_{1j}u_{1i}=\sum_{\pi,\sigma\in P_2(4)}\bar{\delta}_\sigma(a,b,j,i)W_{4N}(\pi,\sigma)$$

Since $P_2(4)=\{\cap\cap,\Cap,\cap\!\!\cap\}$, the Weingarten matrix on the right is given by:
$$W_{4N}=\begin{pmatrix}N^2&N&N\\ N&N^2&N\\ N&N&N^2\end{pmatrix}^{-1}=\frac{1}{N(N-1)(N+2)}\begin{pmatrix}N+1&-1&-1\\ -1&N+1&-1\\ -1&-1&N+1\end{pmatrix}$$

We conclude that we have the following formula:
$$<x_ax_b,x_ix_j>=\frac{1}{N(N+2)}\sum_{\sigma\in P_2(4)}\bar{\delta}_\sigma(a,b,j,i)$$

The matrix on the right, taken with indices $a\leq b$ and $i\leq j$, is then invertible. Thus the variables $x_ax_b$ are linearly independent, as claimed.

(2) Here the scalar products that we are interested in are:
$$<x_ax_bx_c,x_ix_jx_k>=\int_{\bar{O}_N^*}u_{1a}u_{1b}u_{1c}u_{1k}u_{1j}u_{1i}=\sum_{\pi,\sigma\in P_2^*(6)}\bar{\delta}_\sigma(a,b,c,k,j,i)W_{6N}(\pi,\sigma)$$

The set $P_2^*(6)\simeq P_2^*(3,3)$ is by definition formed by the following pairings:
$$
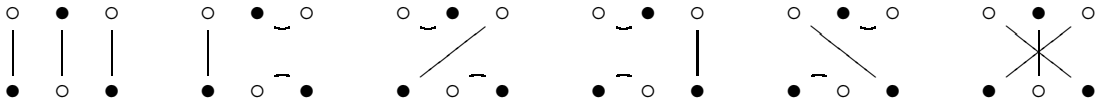
\xymatrix@R=6mm@C=2mm{\circ\ar@{-}[d]&\bullet\ar@{-}[d]&\circ\ar@{-}[d]\\ \bullet&\circ&\bullet}\qquad
\xymatrix@R=6mm@C=2mm{\circ\ar@{-}[d]&\bullet\ar@/_/@{-}[r]&\circ\\ \bullet&\circ\ar@/^/@{-}[r]&\bullet}\qquad
\xymatrix@R=6mm@C=2mm{\circ\ar@/_/@{-}[r]&\bullet&\circ\ar@{-}[dll]\\ \bullet&\circ\ar@/^/@{-}[r]&\bullet}\qquad
\xymatrix@R=6mm@C=2mm{\circ\ar@/_/@{-}[r]&\bullet&\circ\ar@{-}[d]\\ \bullet\ar@/^/@{-}[r]&\circ&\bullet}\qquad
\xymatrix@R=6mm@C=2mm{\circ\ar@{-}[drr]&\bullet\ar@/_/@{-}[r]&\circ\\ \bullet\ar@/^/@{-}[r]&\circ&\bullet}\qquad
\xymatrix@R=6mm@C=2mm{\circ\ar@{-}[drr]&\bullet\ar@{-}[d]&\circ\ar@{-}[dll]\\ \bullet&\circ&\bullet}$$

Now observe that the scalar products of each of these pairings with all the 6 pairings are always, up to a permutation of the terms, $N^3,N^2,N^2,N^2,N,N$. Thus the Gram matrix is stochastic, $G_{6N}\xi=\xi$, where $\xi=(1,\ldots,1)^t$ is the all-one vector. Thus we have $W_{6N}\xi=W_{6N}G_{6N}\xi=\xi$, and so the Weingarten matrix is stochastic too. We conclude that, up to a universal constant depending only on $N$, we have:
$$<x_ax_bx_c,x_ix_jx_k>\sim\sum_{\sigma\in P_2^*(6)}\bar{\delta}_\sigma(a,b,c,k,j,i)$$

Now by computing the rank of this matrix, taken with indices $a\leq c$ and $i\leq k$, we obtain that the variables $x_ax_bx_c$ are linearly independent, as claimed.
\end{proof}

We can formulate our quantum isometry group result, as follows:

\begin{theorem}
The quantum isometry groups of the spheres $\bar{S}^{N-1}_\mathbb R\subset\bar{S}^{N-1}_{\mathbb R,*}\subset S^{N-1}_{\mathbb R,+}$ are the twisted orthogonal quantum groups, $\bar{O}_N\subset\bar{O}_N^*\subset O_N^+$.
\end{theorem}

\begin{proof}
This is known and trivial for $S^{N-1}_{\mathbb R,+}$, and for $\bar{S}^{N-1}_\mathbb R$ it can be deduced as in the proof of Theorem 3.12, by adding signs where needed, and using Proposition 7.9 (1). 

Regarding now $\bar{S}^{N-1}_{\mathbb R,*}$, with $G\subset O_N^+$ and $X_i=\sum_au_{ia}\otimes x_a$, we have the following formula, obtained by using the defining relations for $\bar{S}^{N-1}_{\mathbb R,*}$:
\begin{eqnarray*}
X_iX_jX_k
&=&\sum_{a<c,b\neq a,c}(u_{ia}u_{jb}u_{kc}-u_{ic}u_{jb}u_{ka})\otimes x_ax_bx_c\\
&+&\sum_{a\neq c}(u_{ia}u_{ja}u_{kc}+u_{ic}u_{ja}u_{ka})\otimes x_a^2x_c\\
&+&\sum_{ab}u_{ia}u_{jb}u_{ka}\otimes x_ax_bx_a
\end{eqnarray*}

By interchanging $i\leftrightarrow k$, we have as well a similar formula for $X_kX_jX_i$. Now by using Proposition 7.9 (2), we conclude that the coaction relations $X_iX_jX_k=\pm X_kX_jX_i$ are equivalent to the following system of equations, where $[u_{ia},u_{jb},u_{kc}]=u_{ia}u_{jb}u_{kc}\pm u_{kc}u_{jb}u_{ia}$, with the $\pm$ signs being those making $[u_{ia},u_{jb},u_{kc}]=0$ for the coordinates of $\bar{O}_N^*$:

(1) $[u_{ia},u_{jb},u_{kc}]=[u_{ka},u_{jb},u_{ic}]$, for $a,b,c$ distinct.

(2) $[u_{ia},u_{ja},u_{kc}]=[u_{ka},u_{ja},u_{ic}]$.

(3) $[u_{ia},u_{jb},u_{ka}]=0$.

It is routine to check that these equations are in fact equivalent to $[u_{ia},u_{jb},u_{kc}]=0$, regardless of the indices $i,j,k$ and $a,b,c$, and this gives the result.
\end{proof}

As an application, we can now integrate over the twisted spheres:

\begin{theorem}
Consider the canonical trace $tr:C(\bar{S}^{N-1}_{\mathbb R,\times})\to\mathbb C$, obtained as $tr=I\pi$, where $\pi(x_i)=u_{1i}$, and where $I$ is the Haar integration over $\bar{O}_N^\times$.
\begin{enumerate}
\item $tr$ satisfies $(I\otimes id)\Phi=tr(.)1$, where $\Phi$ is the coaction map.

\item $tr$ is the unique positive unital trace satisfying $(id\otimes tr)\Phi(x)=tr(x)1$.

\item $\sqrt{N}x_i$ is asymptotically real Gaussian/symmetrized Rayleigh/semicircular.
\end{enumerate}
\end{theorem}

\begin{proof}
Here (1) and (2) follow as in the untwisted case, by adding signs where needed. Regarding now (3), the twisted Weingarten computation is as follows:
\begin{eqnarray*}
\int_{\bar{S}^{N-1}_{\mathbb R,\times}}x_i^k
&=&\sum_{\pi,\sigma}\bar{\delta}_\pi(1\ldots 1)\bar{\delta}_\sigma(1\ldots 1)W_{kN}(\pi,\sigma)\\
&\sim&N^{-k/2}\sum_\pi\bar{\delta}_\pi(1\ldots1)^2=N^{-k/2}\#(P_2^\times(k))
\end{eqnarray*}

Thus we obtain the same laws as in the untwisted case, as stated. 
\end{proof}

\section{Polygonal spheres}

We have so far $3+2=5$ noncommutative spheres, and one interesting question is that of finding a suitable axiomatic framework for them. A natural idea here is that of further enlarging our set of spheres, by taking intersections between them, with the intersection operation being obtained by merging the corresponding sets of algebraic relations. 

With the convention, from now on, that the arrows denote inclusions, we have:

\begin{proposition}
The $5$ main spheres, and the intersections between them, are
$$\xymatrix@R=12mm@C=12mm{
S^{N-1}_\mathbb R\ar[r]&S^{N-1}_{\mathbb R,*}\ar[r]&S^{N-1}_{\mathbb R,+}\\
S^{N-1,1}_\mathbb R\ar[r]\ar[u]&S^{N-1,1}_{\mathbb R,*}\ar[r]\ar[u]&\bar{S}^{N-1}_{\mathbb R,*}\ar[u]\\
S^{N-1,0}_\mathbb R\ar[r]\ar[u]&\bar{S}^{N-1,1}_\mathbb R\ar[r]\ar[u]&\bar{S}^{N-1}_\mathbb R\ar[u]}$$
where $\dot{S}^{N-1,d-1}_{\mathbb R,\times}\subset\dot{S}^{N-1}_{\mathbb R,\times}$ is obtained by assuming $x_{i_0}\ldots x_{i_d}=0$, for $i_0,\ldots,i_d$ distinct.
\end{proposition}

\begin{proof}
We must prove that the 4-diagram obtained by intersecting the 5 main spheres coincides with the 4-diagram appearing at bottom left in the statement:
$$\xymatrix@R=13mm@C=13mm{
S^{N-1}_\mathbb R\cap\bar{S}^{N-1}_{\mathbb R,*}\ar[r]&S^{N-1}_{\mathbb R,*}\cap\bar{S}^{N-1}_{\mathbb R,*}\\
S^{N-1}_\mathbb R\cap\bar{S}^{N-1}_\mathbb R\ar[r]\ar[u]&S^{N-1}_{\mathbb R,*}\cap\bar{S}^{N-1}_\mathbb R\ar[u]}
\ \ \xymatrix@R=7mm@C=1mm{&\\=\\&}
\xymatrix@R=13mm@C=13mm{
S^{N-1,1}_\mathbb R\ar[r]&S^{N-1,1}_{\mathbb R,*}\\
S^{N-1,0}_\mathbb R\ar[r]\ar[u]&\bar{S}^{N-1,1}_\mathbb R\ar[u]}$$

But this is clear, because combining the commutation and anticommutation relations leads to the vanishing relations defining spheres of type $\dot{S}^{N-1,d-1}_{\mathbb R,\times}$. More precisely:

(1) $S^{N-1}_\mathbb R\cap\bar{S}^{N-1}_\mathbb R$ consists of the points $x\in S^{N-1}_\mathbb R$ satisfying $x_ix_j=-x_jx_i$ for $i\neq j$. Since $x_ix_j=x_jx_i$, this relation reads $x_ix_j=0$ for $i\neq j$, which means $x\in S^{N-1,0}_\mathbb R$.

(2) $S^{N-1}_\mathbb R\cap\bar{S}^{N-1}_{\mathbb R,*}$ consists of the points $x\in S^{N-1}_\mathbb R$ satisfying $x_ix_jx_k=-x_kx_jx_i$ for $i,j,k$ distinct. Once again by commutativity, this relation is equivalent to $x\in S^{N-1,1}_\mathbb R$.

(3) $S^{N-1}_{\mathbb R,*}\cap\bar{S}^{N-1}_\mathbb R$ is obtained from $\bar{S}^{N-1}_\mathbb R$ by imposing to the standard coordinates the half-commutation relations $abc=cba$. On the other hand, we know from $\bar{S}^{N-1}_\mathbb R\subset \bar{S}^{N-1}_{\mathbb R,*}$ that the standard coordinates on $\bar{S}^{N-1}_\mathbb R$ satisfy $abc=-cba$ for $a,b,c$ distinct, and $abc=cba$ otherwise. Thus, the relations brought by intersecting with $S^{N-1}_{\mathbb R,*}$ reduce to the relations $abc=0$ for $a,b,c$ distinct, and so we are led to the sphere $\bar{S}^{N-1,1}_\mathbb R$.

(4) $S^{N-1}_{\mathbb R,*}\cap\bar{S}^{N-1}_{\mathbb R,*}$ is obtained from $\bar{S}^{N-1}_{\mathbb R,*}$ by imposing the relations $abc=-cba$ for $a,b,c$ distinct, and $abc=cba$ otherwise. Since we know that $abc=cba$ for any $a,b,c$, the extra relations reduce to $abc=0$ for $a,b,c$ distinct, and so we are led to $S^{N-1,1}_{\mathbb R,*}$.
\end{proof}

In order to find now a suitable axiomatic framework for the 9 spheres, we use the following definition, coming from the various formulae in sections 2 and 7:

\begin{definition}
Given variables $x_1,\ldots,x_N$, any permutation $\sigma\in S_k$ produces two collections of relations between these variables, as follows:
\begin{enumerate}
\item Untwisted relations: $x_{i_1}\ldots x_{i_k}=x_{i_{\sigma(1)}}\ldots x_{i_{\sigma(k)}}$, for any $i_1,\ldots,i_k$.

\item Twisted relations: $x_{i_1}\ldots x_{i_k}=\varepsilon\left(\ker(^{\,\,\,i_1\ \,\ldots\ \,i_k}_{i_{\sigma(1)}\ldots i_{\sigma(k)}})\right)x_{i_{\sigma(1)}}\ldots x_{i_{\sigma(k)}}$, for any $i_1,\ldots,i_k$.
\end{enumerate}
The untwisted relations are denoted $\mathcal R_\sigma$, and the twisted ones are denoted $\bar{\mathcal R}_\sigma$.
\end{definition}

Observe that the relations $\mathcal R_\sigma$ are trivially satisfied for the standard coordinates on $S^{N-1}_\mathbb R$, for any $\sigma\in S_k$. A twisted analogue of this fact holds, in the sense that the standard coordinates on $\bar{S}^{N-1}_\mathbb R$ satisfy the relations $\bar{\mathcal R}_\sigma$, for any $\sigma\in S_k$. Indeed, by anticommutation we must have a formula of type $x_{i_1}\ldots x_{i_k}=\pm x_{i_{\sigma(1)}}\ldots x_{i_{\sigma(k)}}$, and the sign $\pm$ obtained in this way is precisely the one given above, $\pm=\varepsilon\left(\ker(^{\,\,\,i_1\ \,\ldots\ \,i_k}_{i_{\sigma(1)}\ldots i_{\sigma(k)}})\right)$.

We have now all the needed ingredients for axiomatizing the various spheres:

\begin{definition}
We have $3$ types of noncommutative spheres $S\subset S^{N-1}_{\mathbb R,+}$, as follows:
\begin{enumerate}
\item Untwisted: $S^{N-1}_{\mathbb R,E}$, with $E\subset S_\infty$, obtained via the relations $\{\mathcal R_\sigma|\sigma\in E\}$.

\item Twisted: $\bar{S}^{N-1}_{\mathbb R,F}$, with $F\subset S_\infty$, obtained via the relations $\{\bar{\mathcal R}_\sigma|\sigma\in F\}$.

\item Polygonal: $S^{N-1}_{\mathbb R,E,F}=S^{N-1}_{\mathbb R,E}\cap\bar{S}^{N-1}_{\mathbb R,F}$, with $E,F\subset S_\infty$.
\end{enumerate}
\end{definition}

Observe that ``untwisted'' means precisely ``monomial'', in the sense of section 2 above. As examples, $S^{N-1}_\mathbb R,S^{N-1}_{\mathbb R,*},S^{N-1}_{\mathbb R,+}$ are untwisted, $\bar{S}^{N-1}_\mathbb R,\bar{S}^{N-1}_{\mathbb R,*},S^{N-1}_{\mathbb R,+}$ are twisted, and the 9 spheres in Proposition 8.1 above are all polygonal. Observe also that the set of polygonal spheres is closed under intersections, due to the following formula:
$$S^{N-1}_{\mathbb R,E,F}\cap S^{N-1}_{\mathbb R,E',F'}=S^{N-1}_{\mathbb R,E\cup E',F\cup F'}$$

Let us try now to understand the structure of the various types of spheres:

\begin{proposition}
The various spheres can be parametrized by groups, as follows:
\begin{enumerate}
\item Untwisted case: $S^{N-1}_{\mathbb R,G}$, with $G\subset S_\infty$ filtered group.

\item Twisted case: $\bar{S}^{N-1}_{\mathbb R,H}$, with $H\subset S_\infty$ filtered group.

\item Polygonal case: $S^{N-1}_{\mathbb R,G,H}$, with $G,H\subset S_\infty$ filtered groups.
\end{enumerate}
\end{proposition}

\begin{proof}
Here (1) is from section 2 above, (2) follows similarly, by taking $H\subset S_\infty$ to be the set of permutations $\sigma\in S_\infty$ having the property that the relations $\bar{\mathcal R}_\sigma$ hold for the standard coordinates, and (3) follows from (1,2), by taking intersections.
\end{proof}

Let us write now the 9 main polygonal spheres as in Proposition 8.4 (3). We say that a polygonal sphere parametrization $S=S^{N-1}_{\mathbb R,G,H}$ is ``standard'' when both filtered groups $G,H\subset S_\infty$ are chosen to be maximal. In this case, Proposition 8.4 (3) and its proof tell us that $G,H$ encode all the monomial relations which hold in $S$.

We have the following result, extending some previous findings from section 2:

\begin{theorem}
The standard parametrization of the $9$ main spheres is
$$\xymatrix@R=10mm@C=10mm{
S_\infty\ar@{.}[d]&S_\infty^*\ar@{.}[d]&\{1\}\ar@{.}[d]&G/H\\
S^{N-1}_\mathbb R\ar[r]&S^{N-1}_{\mathbb R,*}\ar[r]&S^{N-1}_{\mathbb R,+}&\{1\}\ar@{.}[l]\\
S^{N-1,1}_\mathbb R\ar[r]\ar[u]&S^{N-1,1}_{\mathbb R,*}\ar[r]\ar[u]&\bar{S}^{N-1}_{\mathbb R,*}\ar[u]&S_\infty^*\ar@{.}[l]\\
S^{N-1,0}_\mathbb R\ar[r]\ar[u]&\bar{S}^{N-1,1}_\mathbb R\ar[r]\ar[u]&\bar{S}^{N-1}_\mathbb R\ar[u]&S_\infty\ar@{.}[l]}$$
so these spheres come from the $3\times 3=9$ pairs of groups among $\{1\}\subset S_\infty^*\subset S_\infty$.
\end{theorem}

\begin{proof}
The fact that we have parametrizations as above is known to hold for the 5 untwisted and twisted spheres, and for the remaining 4 spheres, this follows by intersecting. In order to prove now that the parametrizations are standard, we must compute the following two filtered groups, and show that we get the groups in the statement:
$$G=\{\sigma\in S_\infty|{\rm the\ relations\ }\mathcal R_\sigma\ {\rm hold\ over\ }S\}$$ 
$$H=\{\sigma\in S_\infty|{\rm the\ relations\ }\bar{\mathcal R}_\sigma\ {\rm hold\ over\ }S\}$$ 

As a first observation, by using the various inclusions between spheres, we just have to compute $G$ for the spheres on the bottom, and $H$ for the spheres on the left:
$$X=S^{N-1,0}_\mathbb R,\bar{S}^{N-1,1}_\mathbb R,\bar{S}^{N-1}_\mathbb R\implies G=S_\infty,S_\infty^*,\{1\}$$
$$X=S^{N-1,0}_\mathbb R,S^{N-1,1}_\mathbb R,S^{N-1}_\mathbb R\implies H=S_\infty,S_\infty^*,\{1\}$$

The results for $S^{N-1,0}_\mathbb R$ being clear, we are left with computing the remaining 4 groups, for the spheres $S^{N-1}_\mathbb R,\bar{S}^{N-1}_\mathbb R,S^{N-1,1}_\mathbb R,\bar{S}^{N-1,1}_\mathbb R$. The proof here goes as follows:

(1) $S^{N-1}_\mathbb R$. According to the definition of $H=(H_k)$, we have:
\begin{eqnarray*}
H_k
&=&\left\{\sigma\in S_k\Big|x_{i_1}\ldots x_{i_k}=\varepsilon\left(\ker(^{\,\,\,i_1\ \,\ldots\ \,i_k}_{i_{\sigma(1)}\ldots i_{\sigma(k)}})\right)x_{i_{\sigma(1)}}\ldots x_{i_{\sigma(k)}},\forall i_1,\ldots,i_k\right\}\\
&=&\left\{\sigma\in S_k\Big|\varepsilon\left(\ker(^{\,\,\,i_1\ \,\ldots\ \,i_k}_{i_{\sigma(1)}\ldots i_{\sigma(k)}})\right)=1,\forall i_1,\ldots,i_k\right\}\\
&=&\left\{\sigma\in S_k\Big|\varepsilon(\tau)=1,\forall\tau\leq\sigma\right\}\end{eqnarray*}

Now since for any $\sigma\in S_k,\sigma\neq1_k$, we can always find a partition $\tau\leq\sigma$ satisfying $\varepsilon(\tau)=-1$, we deduce that we have $H_k=\{1_k\}$, and so $H=\{1\}$, as desired.

(2) $\bar{S}^{N-1}_\mathbb R$. The proof of $G=\{1\}$ here is similar to the proof of $H=\{1\}$ in (1) above, by using the same combinatorial ingredient at the end.

(3) $S^{N-1,1}_\mathbb R$. By definition of $H=(H_k)$, a permutation $\sigma\in S_k$ belongs to $H_k$ when the following condition is satisfied, for any choice of the indices $i_1,\ldots,i_k$:
$$x_{i_1}\ldots x_{i_k}=\varepsilon\left(\ker(^{\,\,\,i_1\ \,\ldots\ \,i_k}_{i_{\sigma(1)}\ldots i_{\sigma(k)}})\right)x_{i_{\sigma(1)}}\ldots x_{i_{\sigma(k)}}$$

When $|\ker i|=1$ this formula reads $x_r^k=x_r^k$, which is true. When $|\ker i|\geq3$ this formula is automatically satisfied as well, because by using the relations $ab=ba$, and $abc=0$ for $a,b,c$ distinct, which both hold over $S^{N-1,1}_\mathbb R$, this formula reduces to $0=0$. Thus, we are left with studying the case $|\ker i|=2$. Here the quantities on the left $x_{i_1}\ldots x_{i_k}$ will not vanish, so the sign on the right must be 1, and we therefore have:
$$H_k=\left\{\sigma\in S_k\Big|\varepsilon(\tau)=1,\forall\tau\leq\sigma,|\tau|=2\right\}$$

Now by coloring the legs of $\sigma$ clockwise $\circ\bullet\circ\bullet\ldots$, the above condition is satisfied when each string of $\sigma$ joins a white leg to a black leg. Thus $H_k=S_k^*$, as desired.

(4) $\bar{S}^{N-1,1}_\mathbb R$. The proof of $G=S_\infty^*$ here is similar to the proof of $H=S_\infty^*$ in (3) above, by using the same combinatorial ingredient at the end.
\end{proof}

We can now formulate a classification result, as follows:

\begin{theorem}
The following hold:
\begin{enumerate}
\item $S^{N-1}_\mathbb R\subset S^{N-1}_{\mathbb R,*}\subset S^{N-1}_{\mathbb R,+}$ are the only untwisted monomial spheres.

\item $\bar{S}^{N-1}_\mathbb R\subset\bar{S}^{N-1}_{\mathbb R,*}\subset S^{N-1}_{\mathbb R,+}$ are the only twisted monomial spheres.

\item The $9$ spheres in Theorem 8.5 are the only polygonal ones.
\end{enumerate}
\end{theorem}

\begin{proof}
By using standard parametrizations, the above 3 statements are equivalent. Now since (1) was proved in section 2 above, all the results hold true. 
\end{proof}

Let us discuss now the computation of the quantum isometry groups of the 9 spheres. The result here, extending previous findings from sections 3 and 7, is as follows:

\begin{theorem}
The quantum isometry groups of the $9$ polygonal spheres are
$$\xymatrix@R=12mm@C=17mm{
O_N\ar[r]&O_N^*\ar[r]&O_N^+\\
H_N\ar[r]\ar[u]&H_N^{[\infty]}\ar[r]\ar[u]&\bar{O}_N^*\ar[u]\\
H_N^+\ar[r]\ar[u]&H_N\ar[r]\ar[u]&\bar{O}_N\ar[u]}$$
where $H_N^+,H_N^{[\infty]}$ and $\bar{O}_N,O_N^*,\bar{O}_N^*,O_N^*$ are noncommutative versions of $H_N,O_N$.
\end{theorem}

\begin{proof}
We already know from sections 3 and 7 that the $O_N$ groups are the correct ones. Regarding the missing 4 computations, those on the bottom left, our precise claim is that we obtain in this way the hyperoctahedral group $H_N$, its free version $H_N^+$, and the ``main'' intermediate liberation $H_N\subset H_N^{[\infty]}\subset H_N^+$, as shown in the diagram above.

Generally speaking, we refer to \cite{ba2} for the proof. In what follows we will only present the main ideas. For the definition and for various technical facts regarding $H_N\subset H_N^{[\infty]}\subset H_N^+$, that we will heavily use in what follows, we refer to \cite{rwe}.

\underline{$S^{N-1,0}_\mathbb R$.} Our sphere here is $S^{N-1,0}_\mathbb R=\mathbb Z_2^{\oplus N}$, formed by the endpoints of the $N$ copies of $[-1,1]$ on the coordinate axes of $\mathbb R^N$. Thus the quantum isometry group is $H_N^+$.

\underline{$S^{N-1,1}_\mathbb R$.} Since the elements $\{x_ix_j|i\leq j\}$ are linearly independent, the trick in \cite{bhg} applies, and gives $G^+(X)\subset O_N$. Now since any affine isometric action $U\curvearrowright S^{N-1,1}_\mathbb R$ must permute the $\binom{N}{2}$ copies of $\mathbb T$ which form our sphere, this gives the result. 

\underline{$\bar{S}^{N-1,1}_\mathbb R$.} By using the maps $\pi_{ij}:C(\bar{S}^{N-1,1}_\mathbb R)\to C(\bar{S}^1_\mathbb R)$ given by $x_k=0$ for $k\neq i,j$, we see that the variables $\{x_ix_j|i\leq j\}$ are once again linearly independent. With this fact in hand, a suitable adaptation of the trick in \cite{bhg} applies, and gives $G^+(X)\subset\bar{O}_N$.

Consider now a quantum subgroup $G\subset\bar{O}_N$. In order to have a coaction map $\Phi:C(\bar{S}^{N-1,1}_\mathbb R)\to C(G)\otimes C(\bar{S}^{N-1,1}_\mathbb R)$, given as usual by $\Phi(x_i)=\sum_au_{ia}\otimes x_a$, the elements $X_i=\sum_au_{ia}\otimes x_a$ must satisfy $X_iX_jX_k=0$, for any $i,j,k$ distinct. We have:
$$X_iX_jX_k=\sum_{ab}(u_{ia}u_{ja}u_{kb}+u_{ja}u_{ka}u_{ib}+u_{ka}u_{ia}u_{jb})\otimes x_a^2x_b$$

Thus, in order for our quantum group $G\subset\bar{O}_N$ to act on $\bar{S}^{N-1,1}_\mathbb R$, its coordinates must satisfy the following relations, for any $i,j,k$ distinct:
$$u_{ia}u_{ja}u_{kb}+u_{ja}u_{ka}u_{ib}+u_{ka}u_{ia}u_{jb}=0$$

By multiplying to the right by $u_{kb}$ and then by summing over $b$, we deduce from this that we have $u_{ia}u_{ja}=0$, for any $i,j$. Now since the quotient of $C(\bar{O}_N)$ by these latter relations is $C(H_N)$, we conclude that we have $G^+(\bar{S}^{N-1,1}_\mathbb R)=H_N$, as claimed.

\underline{$\bar{S}^{N-1,1}_{\mathbb R,*}$.} Let us first prove that $H_N^{[\infty]}$ acts indeed on our sphere. With $X_i=\sum_au_{ia}\otimes x_a$ as usual, and by using the relations for $S^{N-1,1}_{\mathbb R,*}$, we have:
\begin{eqnarray*}
X_iX_jX_k
&=&\sum_{abc}u_{ia}u_{jb}u_{kc}\otimes x_ax_bx_c=\sum_{a,b,c\ not \ distinct}u_{ia}u_{jb}u_{kc}\otimes x_ax_bx_c\\
&=&\sum_{a\neq b}(u_{ia}u_{ja}u_{kb}+u_{ib}u_{ja}u_{ka})\otimes x_a^2x_b\\
&+&\sum_{a\neq b}u_{ia}u_{jb}u_{ka}\otimes x_ax_bx_a+\sum_au_{ia}u_{ja}u_{ka}\otimes x_a^3
\end{eqnarray*}

Now by using various formulae for $H_N^{[\infty]}$, from \cite{rwe}, we obtain, for $i,j,k$ distinct:
$$X_iX_jX_k=\sum_{a\neq b}(0\cdot u_{kb}+u_{ib}\cdot 0)\otimes x_a^2x_b+\sum_{a\neq b}0\otimes x_ax_bx_a+\sum_a(0\cdot u_{ka})\otimes x_a^3=0$$

It remains to prove that we have $X_iX_jX_k=X_kX_jX_i$, for $i,j,k$ not distinct. By replacing $i\leftrightarrow k$ in the above formula of $X_iX_jX_k$, we obtain:
\begin{eqnarray*}
X_kX_jX_i
&=&\sum_{a\neq b}(u_{ka}u_{ja}u_{ib}+u_{kb}u_{ja}u_{ia})\otimes x_a^2x_b\\
&+&\sum_{a\neq b}u_{ka}u_{jb}u_{ia}\otimes x_ax_bx_a+\sum_au_{ka}u_{ja}u_{ia}\otimes x_a^3
\end{eqnarray*}

Let us compare this formula with the above formula of $X_iX_jX_k$. The last sum being 0 in both cases, we must prove that for any $i,j,k$ not distinct and any $a\neq b$ we have:
$$u_{ia}u_{ja}u_{kb}+u_{ib}u_{ja}u_{ka}=u_{ka}u_{ja}u_{ib}+u_{kb}u_{ja}u_{ia}$$
$$u_{ia}u_{jb}u_{ka}=u_{ka}u_{jb}u_{ia}$$

By symmetry the three cases $i=j,i=k,j=k$ reduce to two cases, $i=j$ and $i=k$. The case $i=k$ being clear, we are left with the case $i=j$, where we must prove:
$$u_{ia}u_{ia}u_{kb}+u_{ib}u_{ia}u_{ka}=u_{ka}u_{ia}u_{ib}+u_{kb}u_{ia}u_{ia}$$
$$u_{ia}u_{ib}u_{ka}=u_{ka}u_{ib}u_{ia}$$

By using $a\neq b$, the first equality reads $u_{ia}^2u_{kb}+0\cdot u_{ka}=u_{ka}\cdot 0+u_{kb}u_{ia}^2$, and since we have $u_{ia}^2u_{kb}=u_{kb}u_{ia}^2$, we are done. As for the second equality, this reads $0\cdot u_{ka}=u_{ka}\cdot 0$, which is true as well, and this ends the proof. Finally, regarding the proof of the universality of the action that we constructed, which is quite technical, we refer here to \cite{ba2}.
\end{proof}

\end{document}